\documentclass[a4paper, 12pt]{article}

\usepackage{amsmath} 
\usepackage{theorem}
\usepackage{amssymb}
\usepackage{amsfonts}
\usepackage{latexsym}
\usepackage{amscd}
\usepackage{diagramb}
\input GrCalc3.sty
\usepackage{graphicx}
\usepackage{enumerate}

\DeclareMathAlphabet{\mathpzc}{OT1}{pzc}{m}{it}

\newcommand{\otb}{{\overline{\otimes}}}

\newcommand{\op}{\rm{op}}
\newcommand{\rev}{\rm{rev}}

\newcommand{\mo}{{\mathcal M}}

\newcommand\Rep{\operatorname{Rep}}

\newcommand{\ca}{{\mathcal C}}
\newcommand{\Do}{{\mathcal D}}

\usepackage{xspace,colortbl}
\usepackage{color}

\definecolor{verde}{rgb}{0.,0.7,0.}
\definecolor{indigo}{rgb}{.18, .34, .78}
\definecolor{indigo1}{rgb}{.18, .24, .78}
\definecolor{indigo2}{rgb}{.18, .14, .78}
\definecolor{indigo3}{rgb}{.18, 0., .78}
\definecolor{rojo}{rgb}{1,0,0}
\definecolor{negro}{rgb}{0,0,0}
\definecolor{lila}{rgb}{.46, .16, .78}
\definecolor{lila1}{rgb}{.46, .16, .86}
\definecolor{lila2}{rgb}{.56, .16, .86}
	\definecolor{lila3}{rgb}{.63, .16, .78}
\definecolor{lila4}{rgb}{.7, .16, .78}
\definecolor{lila5}{rgb}{.78, .26, .78}
\definecolor{lila6}{rgb}{.6, 0., .78}

\theoremstyle{plain}
\theoremheaderfont{\normalfont\bfseries}

\newtheorem{thm}{Theorem}[section]
\newtheorem{claim}[thm]{Claim}
\newtheorem{lma}[thm]{Lemma}
\newtheorem{cor}[thm]{Corollary}

\newtheorem{defnlma}[thm]{Definition and Lemma}

\theorembodyfont{\rmfamily}

\newtheorem{rem}[thm]{Remark}
\newtheorem{prop}[thm]{Proposition}

\newcommand{\qed}{\hfill\quad\fbox{\rule[0mm]{0,0cm}{0,0mm}}  \par\bigskip}

\newcommand{\x}{\mbox{-}}
\newcommand{\R}{{\mathcal R}}

\newcommand{\w}{\hspace{-0,12cm}}
\newcommand{\Br}{{\rm Br}}

\def\congo#1{\smash{\mathop{\cong}\limits^{#1}}}
\def\equal#1{\smash{\mathop{=}\limits^{#1}}}
\newcommand{\Del}{\boxtimes}
\newcommand{\comp}{\circ}
\newcommand{\iso}{\cong}
\newcommand{\ot}{\otimes}

\newcommand{\C}{{\mathcal C}}
\newcommand{\M}{{\mathcal M}}
\newcommand{\D}{{\mathcal D}}
\newcommand{\F}{{\mathcal F}}
\newcommand{\G}{{\mathcal G}}

\newcommand{\HH}{{\mathcal H}}
\newcommand{\N}{{\mathcal N}}
\newcommand{\A}{{\mathcal A}}

\newcommand{\E}{{\mathcal E}}

\def\dul#1{\underline{\underline{#1}}}
\def\u{\underline}
\def\o{\overline}
\newcommand{\coev}{\rm coev}
\newcommand{\Ev}{\rm Ev}

\newcommand{\Ll}{{\mathcal L}}
\newcommand{\Pp}{{\mathcal P}}

\newcommand{\YD}{{\mathcal YD}}

\newcommand{\crta}{\overline}
\newcommand{\ev}{\rm ev}
\newcommand{\db}{\rm db}

\newcommand{\Id}{\operatorname {Id}}
\newcommand{\id}{\operatorname {id}}

\newcommand{\Ker}{\operatorname {Ker}}
\newcommand{\Hom}{\operatorname {Hom}}
\newcommand{\End}{\operatorname {End}}
\newcommand{\Fun}{\operatorname {Fun}}

\def\Nn{{\mathbb N}}

\def\CC{{\mathbb C}}
\def\RR{{\mathbb R}}
\def\Zz{{\mathbb Z}}  

\newcommand{\BM}{\operatorname {BM}}
\newcommand{\Pic}{\operatorname{Pic}}
\newcommand{\BrPic}{\operatorname{BrPic}}

\newcommand{\Gal}{\operatorname{Gal}}

\newcommand{\Mod}{\operatorname{Mod}}
\newcommand{\Bimod}{\operatorname{Bimod}}
\newcommand{\Inv}{\operatorname{Inv}}

\newcommand{\BW}{\operatorname {BW}}
\newcommand{\im}{{\rm Im}\,}

\newcommand{\clabel}[1]{\label{c:#1}}
\newcommand{\cref}[1]{Cor.~\ref{c:#1}}

\newcommand{\rlabel}[1]{\label{r:#1}}

\newcommand{\lelabel}[1]{\label{le:#1}}
\newcommand{\leref}[1]{Lemma~\ref{le:#1}}
\newcommand{\eqlabel}[1]{\label{eq:#1}}
\newcommand{\equref}[1]{(\ref{eq:#1})}
\newcommand{\thlabel}[1]{\label{th:#1}}
\newcommand{\thref}[1]{Theorem~\ref{th:#1}}

\newcommand{\prlabel}[1]{\label{pr:#1}}
\newcommand{\prref}[1]{Proposition~\ref{pr:#1}}
\newcommand{\colabel}[1]{\label{co:#1}}

\newcommand{\rmlabel}[1]{\label{rm:#1}}
\newcommand{\rmref}[1]{Remark~\ref{rm:#1}}
\newcommand{\selabel}[1]{\label{se:#1}}
\newcommand{\seref}[1]{Section~\ref{se:#1}}
\newcommand{\sslabel}[1]{\label{ss:#1}}
\newcommand{\ssref}[1]{Subsection~\ref{ss:#1}}

\begin{document}

\title{Villamayor-Zelinsky sequence for symmetric finite tensor categories}
\author{Bojana Femi\'c \vspace{6pt} \\
{\small Facultad de Ingenier\'ia, \vspace{-2pt}}\\
{\small  Universidad de la Rep\'ublica} \vspace{-2pt}\\
{\small  Julio Herrera y Reissig 565, \vspace{-2pt}}\\
{\small  11 300 Montevideo, Uruguay}}

\date{}
\maketitle
\begin{abstract}
We prove that if a finite tensor category $\C$ is symmetric, then the monoidal category of one-sided $\C$-bimodule categories is symmetric. Consequently, the Picard group of $\C$
(the subgroup of the Brauer-Picard group introduced by Etingov-Nikshych-Gelaki) is abelian in this case. We then introduce a cohomology over such $\C$. An important piece of tool
for this construction is the computation of dual objects for bimodule categories and the fact that for invertible one-sided $\C$-bimodule categories the evaluation functor involved
is an equivalence, being the coevaluation functor its quasi-inverse, as we show. Finally, we construct an infinite exact sequence a la Villamayor-Zelinsky for $\C$. It
consists of the corresponding cohomology groups evaluated at three types of coefficients which repeat periodically in the sequence.
\bigbreak

\end{abstract}

\section{Introduction}

The classical Brauer group of a field (introduced in 1929) classifies finite dimensional division algebras, in the way that it consists of equivalence classes of
central simple algebras. Its first generalization is over a commutative ring $R$, where one deals with algebras $A$ satisfying: $A\ot_RA^{op}\iso\End_R(A)$.
The list of further generalizations culminates with the Brauer group of a braided monoidal category $\C$ in \cite{VZ1} in 1998, which consists of equivalence classes of
algebras $A\in\C$ such that $A\ot A^{op}\iso[A,A]$ and $A^{op}\ot A\iso[A,A]^{op}$, where $[A,A]$ denotes the inner hom (such algebras are called Azumaya algebras).
In the following decade various results
have been made concerning the latter group and its subgroups. In 2009, \cite{ENO}, the Brauer-Picard group of a fusion category $\C$ was introduced, it consists
of equivalence classes of semisimple invertible $\C$-bimodule categories $\M$ satisfying: $\M\Del_{\C}\M^{op}\simeq\C$ (and equivalently $\M^{op}\Del_{\C}\M\simeq\C$).
Similarly, the Brauer-Picard group of a {\em finite tensor} category consists of equivalence classes of {\em exact} invertible $\C$-bimodule categories.
This group is involved in the classification of extensions of a given tensor category by a finite group, though it also has relations to mathematical physics,
like rational Conformal Field Theory and 3-dimensional Topological Field Theory, \cite{fsv, KK}.
In \cite{ENO} it is also shown that $\BrPic(Vec_k)\iso\Br(k)$, i.e. the Brauer-Picard group of the category of $k$-vector spaces is isomorphic to the Brauer
group of the field $k$. When $\C$ is braided $\BrPic(\C)$ admits a subgroup $\Pic(\C)$ - the Picard group of $\C$ in which the invertible
categories are one-sided $\C$-module categories (and the action on the other side is induced via the braiding). From \cite{DN} it was known that $\Pic(\C)$
is isomorphic to the Brauer group of {\em exact} Azumaya algebras, and in \cite[Proposition 3.7]{DZ} it is proved that for $\C=\Rep H$, the category of
finite-dimensional representations of a finite-dimensional quasi-triangular Hopf algebra $(H, \R)$, any Azumaya algebra in $\Rep H$ is exact. This proves
that the two groups, different in nature, are isomorphic: $\Pic({}_H\M)\iso\Br(\Rep H)$, where $(H, \R)$ is quasitriangular.

The Brauer group $\Br(k)$ of a field $k$ has a nice cohomological interpretation: it is isomorphic
				to the second Galois cohomology group with respect to the separable closure of the field.
				In the root of this description lies the Crossed Product Theorem relating the relative
				Brauer group $\Br(l/k)$ with the second Galois cohomology group with respect to the
				Galois field extension $l/k$. This cohomological interpretation is possible to transmit
				from the relative to the full	Brauer group because every central simple algebra can be
				split by a Galois field extension. However, this is not the case if we consider
				Galois extensions of commutative rings, not every Azumaya algebra (over a ring) can be
				split by a Galois (ring) extension. Though, instead of the Crossed Product Theorem for
				the relative Brauer group we have
				an infinite exact sequence due to Villamayor and Zelinsky \cite{ViZ} involving cohomology groups for an extension of a commutative ring.
These groups are evaluated at three types of coefficients and the three types of cohomology groups appear periodically in the sequence. If the ring extension is faithfully flat,
the relative Brauer group embeds into the middle term on the second level of the sequence. If the extension is faithfully projective, one has an isomorphism, recovering thus
the Crossed Product Theorem.

In \cite{th} we constructed a version of the above mentioned infinite exact sequence for a commutative bialgebroid (a coring which is an
algebra that fulfills certain compatibility conditions) and we interpreted the middle terms on the first three levels of the sequence.
If $R\to S$ is a commutative ring extension, then 
$S\ot_R S$ is a commutative bialgebroid 
and the new sequence generalizes the previous one.

\medskip

In the present paper we introduce a cohomology for a symmetric finite tensor category $\C$, which we still call Amitsur cohomology.
We specialize it in three types of coefficients, one of which is the Picard group of $\C$ (the subgroup of the Brauer-Picard group).
For this construction we previously prove the following two properties. Firstly, that for a symmetric finite tensor category $\C$ the monoidal 
category of one-sided $\C$-bimodule categories $(\C^{br}\x\Mod, \Del_{\C}, \C)$, which is obtained by truncating the corresponding 2-category, 
is symmetric monoidal. This should not be a surprise, because, as we cited above: $\Pic(\Rep H)\iso\Br(\Rep H)$
for any finite-dimensional quasi-triangular Hopf algebra $H$. It is known that for any symmetric
monoidal category the Brauer group of it is abelian. Therefore, when $H$ is triangular, i.e. $\Rep H$ is symmetric, we get that $\Pic(\Rep H)$ is abelian.
Secondly, we prove that the dual object for an invertible one-sided $\C$-bimodule category is its opposite category and that the evaluation
functor involved is an equivalence, whereas the coevaluation functor is its quasi-inverse.
We then construct an infinite exact sequence a la Villamayor-Zelinsky for $\C$, which consists of the corresponding cohomology groups
evaluated at the mentioned three types of coefficients which repeat periodically in the sequence. The interpretation of some of its
terms we develop in the forthcoming paper \cite{F-VZ2}.

It is known that any symmetric finite tensor category is equivalent to the category
$\Rep H$ of finite-dimensional representations of a finite-dimensional triangular quasi-Hopf algebra $H$. If $H$ is a Hopf algebra, we
say that $\C$ is {\em strong} (in the style of \cite{NR}). As it was proved 
in \cite[Theorem 5.1.1]{AEG}, \cite[Theorem 4.3]{EG}, any finite-dimensional triangular
Hopf algebra over an algebraically closed field of characteristic zero is the Drinfel'd twist of a modified supergroup algebra. The
corresponding representation categories are equivalent, hence we have that any symmetric finite tensor category $\C$ that is strong
is the representation category of a modified supergroup algebra $\Lambda(V)\rtimes kG$ with a triangular structure $\R$.
Collecting the above-said and the results from \cite[Theorem 6.5]{CuF} and \cite{Car2}, we obtain that the corresponding Picard group decomposes as:
$\Pic(\C)\iso\BM(k,kG,\R) \times \Gal(\Lambda(V);{}_{kG}\M)\iso \BM(k,kG,\R) \times S^2(V^*)^G$.
Here $\BM(k,kG,\R)$ is the Brauer group of $G$-graded vector spaces with respect to the braiding induced by $\R$,
the group $S^2(V^*)^G$ is that of symmetric matrices over $V^*$ invariant under the conjugation by elements of $G$, and $\Gal(\Lambda(V);{}_{kG}\M)$ is
the group of $\Lambda(V)$-Galois objects in ${}_{kG}\M$, which are one-sided comodules over $\Lambda(V)$. In particular, for the Sweeder
and the Nichols Hopf algebras $H_4$ and $E(n)$ we have from \cite{CC2} and \cite{VZ2}:
$$\Pic(\Rep(H_4))\iso \BW(k) \times (k,+)\quad\textnormal{and}\quad \Pic(\Rep(E(n))\iso \BW(k) \times (k,+)^{n(n+1)/2}$$
where $\BW(k)$ denotes the Brauer-Wall group (the corresponding Azumaya algebras are $\Zz_2$-graded). For any finite group $G$ it is
$\Pic(\Rep G)\iso H^2(G, k^{\times})$, \cite[Corollary 8.10]{Gr}.

\medskip

The paper is organized as follows. In the Preliminaries we recall some definitions 
and prove one basic result. In the third section we study the left and right module structures over a Deligne tensor product of two finite tensor categories: $\C\Del\D$
and how they are related to a $\C\x\D$-bimodule structure. We prove in \prref{sim bimod} that the monoidal category $(\C^{br}\x\Mod, \Del_{\C}, \C)$
of one-sided bimodule categories over a symmetric
finite tensor category $\C$ is symmetric. The fourth section is dedicated to the study of dual objects in the monoidal category of bimodule categories over any finite
tensor category. We prove here that the evaluation functor is an equivalence functor for invertible bimodule categories.
In \seref{Amitsur coh} we introduce our Amitsur cohomology over symmetric finite tensor categories.
Our infinite exact sequence a la Villamayor-Zelinsky is constructed in the last section in \thref{VZ}, which is the main result of this paper. Finally,
we record that any symmetric finite tensor category is the representation category of a Drinfel'd twist of a modified supergroup algebra
and how the corresponding Picard group decomposes into the above-mentioned direct product.

\section{Preliminaries and notation}

Throughout $I$ will denote the unit object in a monoidal category $\C$ and $k$ an algebraically closed field of characteristic zero.
When there is no confusion we will denote the identity functor on a category $\M$ by $\M$. We proceed to recall some definitions and basic properties.

A {\em finite category} over $k$ is a $k$-linear abelian category equivalent to a category of finite-dimensional representations of a finite-dimensional $k$-algebra.
A {\em tensor category} over $k$ is a $k$-linear abelian rigid monoidal category such that the unit object is simple. An object $X$ is said to be {\em simple}
if $\End(X)=k\Id_X$.
A {\em finite tensor category} is a tensor category such that the underlying category is finite. For example, if $H$ is a finite-dimensional Hopf algebra
(or, more generally, a finite-dimensional quasi-Hopf algebra) the category of its representations $\Rep H$ is a finite tensor category.

All tensor categories will be assumed to be over a field $k$, all categories will be finite and all  functors will be assumed to be $k$-linear.

We assume the reader is familiar with the notions of a left, right and bimodule categories over a tensor category, (bi)module functors,
Deligne tensor product of finite abelian categories, tensor product of bimodule categories and exact module categories. For the respective definitions we refer to
\cite{EO}, \cite{ENO}, \cite{Gr}, \cite{EGNO}.
The Deligne tensor product bifunctor $-\Del-$ and the action bifunctor for module categories $-\crta\ot-$ are biexact in both variables \cite[Proposition 1.46.2]{EGNO}.

\medskip

For finite tensor categories $\C, \D$ a $\C\x\D$-bimodule category is a left $\C\Del\D^{rev}$-module category and
a right $\C^{rev}\Del\D$-module category. Here $\C^{rev}$ is the category with the tensor product reversed with respect to that of $\C$: $X\ot^{\rev} Y= Y\ot X$,
and the associativity constraint $a^{\rev}_{X,Y,Z}=a^{-1}_{Z,Y,X}$ for $X, Y, Z\in \C.$
A $(\ca, \Do)$-bimodule category is \emph{exact} if it is exact as a left $\ca\boxtimes \Do^{\rev}$-module category,  \cite{ENO}, \cite{Gr}.

\medbreak

{\bf $\C$-balanced functors}. Let $\M$ be a right $\C$-module and $\N$ a left $\C$-module category.
For any abelian category $\A$ a bifunctor $F: \M\times\N\to\A$ additive in every argument is called {\em $\C$-balanced}
if there are natural isomorphisms $b_{M,X,N}: F(M\crta\ot X, N) \stackrel{\iso}{\to} F(M, X\crta\ot N)$ for all $M\in\M, X\in\C, N\in\N$  s.t.
\begin{equation} \eqlabel{C-balanced}
\scalebox{0.84}{
\bfig \hspace{-1cm}
\putmorphism(-200,500)(1,0)[F((M\crta\ot X)\crta\ot Y, N)` F(M \crta\ot(X\ot Y), N)`F(m_{M,X,Y}^r, N)]{1500}{-1}a
\putmorphism(1300,500)(1,0)[\phantom{(X \ot (Y \ot U)) \ot W}`F(M,(X\ot Y)\crta\ot N)` b_{M,X \ot Y,N}]{1350}1a
\putmorphism(2870,500)(0,-1)[``F(M, m_{X,Y,N}^l)]{500}1l
\putmorphism(-160,500)(0,-1)[``b_{M\crta\ot X,Y,N}]{500}1r
\putmorphism(-200,0)(1,0)[F((M\crta\ot X), Y\crta\ot N)` F(M, X \crta\ot(Y \crta\ot N))` b_{M,X,Y \crta\ot N}]{2960}1b
\efig}
\end{equation}
commutes. 
\medskip

A {\bf $\C$-balanced natural transformation} $\Psi: F\to G$ between two $\C$-balanced functors $F, G: \M\times\N\to\A$ with their respective balancing isomorphisms $f_X$ and $g_X$,
is a natural transformation such that the following diagram commutes:
\begin{equation}  \eqlabel{C balanced nat tr} \hspace{-1cm} 
\scalebox{0.88}{\bfig
 \putmorphism(0,400)(1,0)[F((M\crta\ot X), N)` F(M, (X\crta\ot N))`f_{M,X,N}]{1500}1a
 \putmorphism(0,0)(1,0)[G((M\crta\ot X), N) ` G(M, (X\crta\ot N)).` g_{M,X,N}]{1500}1a
\putmorphism(0,400)(0,-1)[\phantom{B\ot B}``\Psi(M\crta\ot X, N)]{380}1l
\putmorphism(1500,400)(0,-1)[\phantom{B\ot B}``\Psi(M, X\crta\ot N)]{380}1r
\efig}
\end{equation}

\medbreak

{\bf Relative tensor product}. Let $\C, \D, \E$ be finite tensor categories. For a $\C\x\D$-bimodule category $\M$ and a $\D\x\E$-bimodule category $\N$ the tensor product over $\D$:
$\M\boxtimes_{\D}\N$ is a $\C\x\E$-bimodule category with a right exact $\D$-balanced functor $\pi_{\M,\N}: \M\Del\N\to \M\Del_{\D}\N$ universal for right exact $\D$-balanced 
functors from $\M\Del\N$. By $\D$-balance of  $\pi_{\M,\N}$ there is an isomorphism $\beta_{M,X,N}: (M\crta\ot X)\Del_{\C}N\to M\Del_{\C}(X\crta\ot N)$ which satisfies: 
\begin{equation} \eqlabel{id C-balanced}
\beta_{M,X,Y \crta\ot N} \beta_{M\crta\ot X,Y,N} ({m_r}_{M,X,Y}\Del_{\C} N) = (M\Del_{\C}{m_l}_{X,Y,N})\beta_{M,X \ot Y,N}
\end{equation}
Given an $\E\x\F$-bimodule category $\Pp$, there is a canonical equivalence of $\C\x\F$-bimodule
categories: $(\M\Del_{\D}\N)\Del_{\E}\Pp\simeq\M\Del_{\D}(\N\Del_{\E}\Pp)$, \cite[Remark 3.6]{ENO}. 

\medbreak

 A right $\ca$-module category $\M$ gives rise to a left $\C$-module category $\mo^{op}$ with the action
given by \equref{left op} and associativity isomorphisms $m^{op}_{X,Y,M}= m_{M,{}^*Y, {}^*X}$ for all $X, Y\in \ca, M\in \mo$.
Similarly, a left $\ca$-module category $\mo$ gives rise to a right $\ca$-module category $\mo^{op}$ with the action 
given via \equref{right op}. Here ${}^*X$ denotes
the left dual object and $X^*$ the right dual object for $X\in\C$. If $\mo$ is a $(\ca,\Do)$-bimodule category then $\mo^{op}$ is a
$(\Do,\ca)$-bimodule category and $(\mo^{op})^{op}\iso\M$ as $(\ca,\Do)$-bimodule categories. \vspace{-1cm}
\begin{center}
\begin{tabular}{p{4.8cm}p{2cm}p{5.8cm}}
\begin{eqnarray}  \eqlabel{left op}
X\crta\ot^{op}M=M\otb {}^*X
\end{eqnarray}  & &
\begin{eqnarray} \eqlabel{right op}
M\crta\ot^{op}X=X^*\crta\ot M
\end{eqnarray}
\end{tabular}
\end{center} \vspace{-0,7cm}

A $(\ca, \Do)$-bimodule category $\mo$ is called \emph{invertible} \cite{ENO} if
there are  equivalences of bimodule categories
$$\mo^{\op}\boxtimes_{\ca} \mo\simeq \Do, \quad \mo\boxtimes_{\Do} \mo^{\op}\simeq \ca.$$


The Brauer-Picard group of a fusion category, introduced in \cite{ENO}, is a group of equivalence classes of semisimple invertible module categories.
In a more general setting, the Brauer-Picard group $\BrPic(\C)$ of a finite tensor category $\C$ is a group of equivalence classes of exact invertible module categories
(a module category over a fusion category is exact if and only if it is semisimple \cite[Example 3.3 (iii)]{EO}).

\medskip

{\bf One-sided $\C$-bimodule categories.}
When $\C$ is braided with a braiding $\Phi$, then every left $\C$-module category is a right and a $\C$-bimodule category:
$M\crta\ot X=X\crta\ot M$ with the isomorphism functors
$m_{M,X,Y}^r: M\crta\ot (X\ot Y)\to(M\crta\ot X)\crta\ot Y$ defined via:
\begin{equation} \eqlabel{right associator}
\scalebox{0.84}{
\bfig 
\putmorphism(0,500)(1,0)[M \crta\ot (X\ot Y)	`(M\crta\ot X)\crta\ot Y` m_{M,X,Y}^r]{2660}1a
\putmorphism(0,0)(1,0)[(X \w\ot\w Y) \crta\ot M` (Y\ot X)\crta\ot M` \Phi_{X,Y}\crta\ot M]{1450}1b
\putmorphism(1400,0)(1,0)[\phantom{(X \ot (Y \ot U))}`Y \crta\ot (X\crta\ot M),` m_{Y,X,M}^l]{1250}1b
\putmorphism(60,500)(0,1)[`` =]{500}1l 
\putmorphism(2670,500)(0,1)[`` =]{500}1r 
\efig}
\end{equation}
see \cite[Section 2.8]{DN}. 
Moreover, the $\C$-bimodule associativity constraint is given by:
\begin{equation} \eqlabel{mixed assoc}
\scalebox{0.84}{
\bfig \hspace{-1cm}
\putmorphism(0,500)(1,0)[(X\crta\ot M) \crta\ot Y	`X\crta\ot (M\crta\ot Y)` a_{X,M,Y}]{3100}1a
\putmorphism(0,0)(1,0)[Y\crta\ot (X\crta\ot M)` (Y\w\ot X)\crta\ot M`(m^l)^{-1}_{Y,X,M}]{1050}1b
\putmorphism(1000,0)(1,0)[\phantom{(X \ot (Y \ot U))}`(X\w\ot Y)\crta\ot M` \Phi_{Y,X}\ot M]{1100}1b
\putmorphism(2100,0)(1,0)[\phantom{(X \ot (Y \ot U))}`X \crta\ot (Y\crta\ot M)` m_{X,Y,M}^l]{1000}1b
\putmorphism(60,500)(0,1)[``=]{500}1l 
\putmorphism(3100,500)(0,1)[``=]{500}{-1}r 
\efig}
\end{equation}
for all $X,Y\in\C, M\in\M$. The $\C$-bimodule categories obtained in this way are called {\em one-sided $\C$-bimodule categories.}

For a braided finite tensor category $\C$ exact invertible one-sided $\C$-bimodule categories form a subgroup of $\BrPic(\C)$, called the
Picard group of $\C$ and denoted by $\Pic(\C)$, \cite[Section 4.4]{ENO}, \cite[Section 2.8]{DN}.

\subsection{Deligne tensor product categories as (braided) monoidal categories} \sslabel{product cats}

We start with the following:

\begin{lma} \lelabel{Weq}
Let $\C$ and $\D$ be finite tensor categories. Given a right $\C$-module category $\M$, a left $\C$-module category $\M'$,
a right $\D$-module category $\N$ and a left $\D$-module category $\N'$, there is an isomorphism of categories:
$$(\M\Del_{\C}\M')\Del(\N\Del_{\D}\N')\iso(\M\Del\N)\Del_{\C\Del\D}(\M'\Del\N').$$
\end{lma}

\begin{proof}
We will use \cite[Remark 3.2]{ENO} for the definition of a balanced functor.
The functor $F:\M\Del\N\Del\M'\Del\N'\to(\M\Del_{\C}\M')\Del(\N\Del_{\D}\N')$ defined via
$F(M\Del N, M'\Del N'):=(M\Del_{\C} M')\Del(N\Del_{\D} N')$ is $\C\Del\D$-balanced. Indeed, the balancing isomorphism is identity:
$F((M\Del N)\crta\ot(C\Del D), M'\Del N')=F((M\crta\ot C)\Del (N\crta\ot D), M'\Del N')=
((M\crta\ot C)\Del_{\C}M')\Del ((N\crta\ot D)\Del_{\D}N')=(M\Del_{\C}(C\crta\ot M'))\Del(N\Del_{\D}(D\crta\ot N'))=
F(M\Del N, (C\Del D)\crta\ot (M'\Del N'))$, and it is clear that the $\C\Del\D$-coherence for $F$ follows.

For the inverse functor observe the following diagram. The equivalences in the right vertical arrows are due to \cite[Lemma 4.1]{Gr}.
The functor $F_3$ is well-defined if $F_2$ is $\C$-balanced, whereas $F_2$ is well-defined if $F_1$ is $\D$-balanced. The check of the last two conditions
is easy and we leave it to the reader.

\hspace{-1cm}
\scalebox{0.86}{
$$
\bfig
\putmorphism(-1730,-120)(1, 0)[\M\Del\M'\Del\N\Del\N'`(\M\Del\M'\Del\N)\Del_{\D}\N'`\Pi_{\D}]{1200}1l
\putmorphism(-2100, -160)(2, -1)[\phantom{A\ot (A\#H)\ot (A\#H)}` `F_1]{2580}1l
\putmorphism(-1730,-120)(0, -1)[`\M\Del\N\Del\M'\Del\N'`]{1400}0r
\putmorphism(-2030,-120)(0, -1)[``\Id\Del t\Del\Id]{1400}1r
\putmorphism(-430,-120)(0,-1)[``\simeq]{350}1r
\putmorphism(-400,-480)(1, 0)[\M\Del((\M'\Del\N)\Del_{\D}\N')` \M\Del_{\C}((\M'\Del\N)\Del_{\D}\N')`\Pi_{\C}]{1400}1a
\putmorphism(1000,-480)(0, -1)[\phantom{A\ot [(A\#H)^A\ot (A\#H)^A]}` \M\Del_{\C}(\M'\Del(\N\Del_{\D}\N'))`\simeq]{350}1r
\putmorphism(1000,-830)(0, -1)[\phantom{A\ot [(A\#H)^A\ot H]}`(\M\Del_{\C}\M')\Del(\N\Del_{\D}\N')`\simeq]{350}1r
\putmorphism(1000,-1180)(0, -1)[\phantom{[A\ot (A\#H)^A]\ot H]}`(\M\Del\N)\Del_{\C\Del\D}(\M'\Del\N') `F_3]{350}1r
\putmorphism(-500,-480)(1, -1)[\phantom{A\ot (A\#H)\ot (A\#H)}` `F_2]{1020}1l
\putmorphism(-1680,-1530)(1, 0)[\phantom{(A\#H)\ot (A\#H)}` \phantom{(\M\Del\N)\Del_{\C\Del\D}(\M'\Del\N') }`\Pi_{\C\Del\D}]{2680}1b
\efig
$$}
\vspace{1,5cm}

\noindent  The functor $F_3$ is given by: $F_3((M\Del_{\C} M')\Del(N\Del_{\D} N'))=(M\Del N)\Del_{\C\Del\D}(M'\Del N')$.
It is straightforward to see that $F$ and $F_3$ are inverse to each other.
\qed\end{proof}

Let $\C, \D$ be finite tensor categories. The category $\C\Del\D$ is a finite tensor category with unit object $I_{\C}\Del I_{\D}$ and componentwise tensor product:
$(X\Del Y)\odot(X'\Del Y')=(X\ot_{\C}X')\Del(Y\ot_{\D}Y')$, where $\ot_{\C}$ denotes the tensor product in $\C$ and similarly for $\D$.
Suppose now that $\C$ and $\D$ are braided with the braidings $\Phi^{\C}$ and $\Phi^{\D}$, respectively.
Then $\C\Del\D$ is braided with the braiding
$$\tilde\Phi_{X\Del Y, X'\Del Y'}: (X\Del Y)\odot(X'\Del Y')\to(X'\Del Y')\odot(X\Del Y)$$
between the objects $X\Del Y, X'\Del Y'\in\C\Del\D$ which is given by:
\begin{equation} \eqlabel{braid two}
\scalebox{0.88}{\bfig
 \putmorphism(-30,500)(1,0)[(X\Del Y)\odot(X'\Del Y')` (X'\Del Y')\odot(X\Del Y) `\tilde\Phi_{X\Del Y, X'\Del Y'}]{1630}1a
 \putmorphism(-50,0)(1,0)[(X\ot_{\C}X')\Del(Y\ot_{\D}Y')` (X'\ot_{\C}X)\Del(Y'\ot_{\D}Y)` \Phi_{X,X'}^{\C}\Del\Phi_{Y,Y'}^{\D}]{1700}1a
\putmorphism(-60,500)(0,-1)[\phantom{B\ot B}``=]{480}1l
\putmorphism(1630,500)(0,-1)[\phantom{B\ot B}``=]{480}1r
\efig}
\end{equation}

\begin{rem} \rmlabel{tensor product copies of C} \rlabel{tensor product copies of C}
Observe that the tensor product $\odot$ can be considered as $\Del_{\C\Del\D}$:
$$(X\Del Y)\odot(X'\Del Y')=(X\ot_{\C}X')\Del(Y\ot_{\D}Y')=(X\Del_{\C}X')\Del(Y\Del_{\D}Y')\iso (X\Del Y)\Del_{\C\Del\D}(X'\Del Y')$$
where the last isomorphism is due to \leref{Weq}.
\end{rem}

\section{Bimodule categories over braided and symmetric tensor categories} \selabel{braided}

Recall that for (non braided) finite tensor categories $\C, \D$ a $\C\x\D$-bimodule category is a left $\C\Del\D^{rev}$-module category and
a right $\C^{rev}\Del\D$-module category, where $\C^{rev}$ is the category with the tensor product reversed with respect to that of $\C$.
One would expect that when $\C$ and $\D$ are braided, that a $\C\x\D$-bimodule category is a one-sided $\C\Del\D$-bimodule category. Let us investigate this.

When $\C$ is braided with a braiding $\Phi$, the category $\C^{rev}$ is braided in two ways. Its braiding is given by
$X\ot^{rev}Y=Y\ot X\stackrel{\Phi^i}{\to}X\ot Y=Y\ot^{rev}X$, for $i=\pm 1$. We will denote the two braided monoidal categories by
$\C^i$ with $i=\pm 1$.

\begin{prop} \prlabel{big associators}
Let $\C$ and $\D$ be finite tensor categories and let $\M$ be a $\C\x\D$-bimodule category. Then we have:
\begin{enumerate} 
\item If $\D$ is braided, then $\M$ is a left $\C\Del\D$-module category with the action given by: $(X\Del Y)\crta\ot M=X\crta\ot M\crta\ot Y$, for
$X\in\C, Y\in\D, M\in\M$ and the left associator $m^L$ defined via:
\begin{equation} \eqlabel{m^L}
\scalebox{0.84}{
\bfig \hspace{-1cm}
\putmorphism(0,600)(1,0)[((X\Del Y)\odot(X'\Del Y')) \crta\ot M	` (X\Del Y)\crta\ot((X'\Del Y')\crta\ot M)` m^L]{1600}1a
\putmorphism(2150,600)(1,0)[`` =]{430}{-1}a
\putmorphism(60,600)(0,1)[`((X\ot X')\Del(Y\ot Y'))\crta\ot M`=]{300}1l
\putmorphism(60,300)(0,1)[``=]{300}1l
\putmorphism(0,0)(1,0)[((X\ot X')\crta\ot M)\crta\ot(Y\ot Y')` ((X\ot X')\crta\ot M)\crta\ot(Y'\ot Y)`\Id\crta\ot\tilde\Phi^i_{Y,Y'}]{1620}1b
\putmorphism(1620,0)(1,0)[\phantom{((X\ot X')\crta\ot M)\crta\ot(Y'\ot Y)}`(X\crta\ot(X'\crta\ot M))\crta\ot(Y'\ot Y) ` m^l\crta\ot\Id]{1480}1b
\putmorphism(3100,300)(0,1)[((X\crta\ot(X'\crta\ot M)))\crta\ot Y')\crta\ot Y``m^r]{300}{-1}r
\putmorphism(3100,600)(0,1)[(X\crta\ot((X'\crta\ot M)\crta\ot Y'))\crta\ot Y``a\crta\ot Y]{300}{-1}r
\efig}
\end{equation}
where $i=\pm 1$.
\item If $\C$ is braided, then $\M$ is a right $\C\Del\D$-module category with the action given by: $M\crta\ot(X\Del Y)=X\crta\ot M\crta\ot Y$ and the
right associator $m^R$:
\begin{equation} \eqlabel{m^R}
\scalebox{0.84}{
\bfig \hspace{-1cm}
\putmorphism(0,600)(1,0)[M\crta\ot((X\Del Y)\odot(X'\Del Y')) ` (M\crta\ot(X\Del Y))\crta\ot(X'\Del Y')` m^R]{1600}1a
\putmorphism(2150,600)(1,0)[`` =]{430}{-1}a
\putmorphism(60,600)(0,1)[`M\crta\ot((X\ot X')\Del(Y\ot Y'))`=]{300}1l
\putmorphism(60,300)(0,1)[``=]{300}1l
\putmorphism(0,0)(1,0)[((X\ot X')\crta\ot M)\crta\ot(Y\ot Y')` ((X\ot X')\crta\ot M)\crta\ot(Y'\ot Y)`\Phi^j_{X,X'}\crta\ot\Id]{1620}1b
\putmorphism(1620,0)(1,0)[\phantom{((X\ot X')\crta\ot M)\crta\ot(Y'\ot Y)}`(X'\crta\ot(X\crta\ot M))\crta\ot(Y\ot Y') ` m^l\crta\ot\Id]{1480}1b
\putmorphism(3100,300)(0,1)[((X'\crta\ot(X\crta\ot M)))\crta\ot Y)\crta\ot Y' ``m^r]{300}{-1}r
\putmorphism(3100,600)(0,1)[(X'\crta\ot((X\crta\ot M)\crta\ot Y))\crta\ot Y' ``a\crta\ot Y]{300}{-1}r
\efig}
\end{equation}
where $j=\pm 1$.
\item If both $\C$ and $\D$ are braided, then $\M$ is a $\C\Del\D$-bimodule category with the bimodule associativity constraint $A$ given by:
\begin{equation} \eqlabel{constr A}
\scalebox{0.84}{
\bfig \hspace{-1cm}
\putmorphism(0,600)(1,0)[((X\Del Y)\crta\ot M)\crta\ot(X'\Del Y')` (X\Del Y)\crta\ot (M\crta\ot(X'\Del Y'))` A]{1600}1a
\putmorphism(2150,600)(1,0)[`` =]{430}{-1}a
\putmorphism(60,600)(0,1)[` (X'\crta\ot(X\crta\ot(M\crta\ot Y)))\crta\ot Y'`=]{300}1l
\putmorphism(60,300)(0,1)[` ((X'\crta\ot X)\crta\ot(M\crta\ot Y))\crta\ot Y' ` (m^l)^{-1}\crta\ot Y']{360}1l
\putmorphism(60,-60)(0,1)[` ((X\crta\ot X')\crta\ot(M\crta\ot Y))\crta\ot Y' ` \Phi^s_{X',X}\crta\ot\Id]{360}1l
\putmorphism(60,-420)(0,1)[`(X'\crta\ot X)\crta\ot((M\crta\ot Y)\crta\ot Y') `a]{300}1l
\putmorphism(620,-720)(1,0)[` (X\ot X')\crta\ot (M\crta\ot(Y\ot Y')) ` \Id\crta\ot(m^r)^{-1}]{1020}1b
\putmorphism(2240,-720)(1,0)[` (X\ot X')\crta\ot (M\crta\ot(Y'\ot Y)) ` \Id\crta\ot\Phi^t_{Y,Y'}]{1020}1b
\putmorphism(3100,-420)(0,1)[((X\ot X')\crta\ot M)\crta\ot(Y'\ot Y)  ` ` a^{-1}]{300}{-1}r
\putmorphism(3100,-60)(0,1)[\phantom{((X\ot X')\crta\ot M)\crta\ot(Y'\ot Y)}` ` m^l\crta\ot\Id]{360}{-1}r
\putmorphism(3100,300)(0,1)[X\crta\ot((X'\crta\ot M)\crta\ot(Y'\ot Y)) ` (X\crta\ot(X'\crta\ot M))\crta\ot(Y'\ot Y)`a]{360}{-1}r
\putmorphism(3100,600)(0,1)[X\crta\ot(((X'\crta\ot M)\crta\ot Y')\crta\ot Y) ``X\crta\ot m^r]{300}{-1}r
\efig}
\end{equation}
where $s, t=\pm 1$.
\end{enumerate}
\end{prop}

\begin{proof}
The proof is direct. One uses the coherences satisfied by the associators $m^l, m^r$ and $a$ of the $\C\x\D$-bimodule structure of $\M$, and
apart from them, the computation reduces to certain conditions on the braidings. In the case of 1) it is: $(Z'\ot\tilde\Phi^i_{X',Y'})\tilde\Phi^i_{X'\ot Y',Z'}=
(\tilde\Phi^i_{Y',Z'}\ot X')\tilde\Phi^i_{X',Y'\ot Z'}$, which is fulfilled by the naturality of the braiding $\tilde\Phi^i$ in the first coordinate.
In the case of 3) one condition is:
$$(Y\ot\Phi^s_{Z,Y})(\Phi^s_{Z,X}\ot Y) \ot (\tilde\Phi^t_{Y',Z'}\ot X') (Y'\ot\tilde\Phi^t_{X',Z'}) \tilde\Phi^i_{X',Y'\ot Z'}=
\Phi^s_{Z,X\ot Y} \ot (Z'\ot\tilde\Phi^i_{X',Y'})\tilde\Phi^t_{X'\ot Y',Z'},$$
which is fulfilled by one of the two axioms for the braiding $\Phi$, on the one hand, and on the other, by the other axiom of the braiding $\tilde\Phi^s$
and the naturality of the braiding $\tilde\Phi^r$ in the first coordinate. The rest of the coherences to check are resolved in a similar fashion.
Note that the claims hold true for arbitrary choices of indices $i,j,s,t\in\{-1, 1\}$.
\qed\end{proof}

From the parts 1) and 2) in the above proposition we see that a $\C\x\D$-bimodule category is a left and a right $\C\Del\D$-module category
with: $(X\Del Y)\crta\ot M=X\crta\ot M\crta\ot Y=M\crta\ot (X\Del Y)$ for $X\in\C, Y\in\D, M\in\M$.
In order to investigate under which conditions a $\C\x\D$-bimodule category is a one-sided $\C\Del\D$-bimodule category, we will consider constraints
$m^L, m^R$ and $A$ from \prref{big associators} starting with their corresponding indices $i, j,s,t\in\{-1, 1\}$ and check the conditions
\equref{right associator} and \equref{mixed assoc}. The braiding appearing in these two conditions will now refer to the braiding $\Psi$ of
$\C\Del\D$ where $\M$ is assumed to be a left $\C\Del\D$-module category. According to \equref{m^L}, $\D$ is considered as a braided category
$\D^i$, thus $\Psi_{X\Del Y, X'\Del Y'}=\Phi_{X,X'}\Del\tilde\Phi^i_{Y,Y'}$, see \equref{braid two}.

\begin{prop} \prlabel{ }
Let $\C$ be a braided finite tensor category, $\D$ a symmetric finite tensor category and $\M$ a $\C\x\D$-bimodule category. Then $\M$
is a one-sided $\C\Del\D$-bimodule category with constraints $m^L, m^R$ and $A$ defined in \prref{big associators}, taking $j=1$ and $s=1$
in the definitions of $m^R$ and $A$, respectively.
\end{prop}

\begin{proof}
The diagram chasing argument yields that the conditions \equref{right associator} and \equref{mixed assoc}, in order for $\M$ to be
a one-sided $\C\Del\D$-bimodule category, reduce to the following:
$$((X'\ot X)\ot\tilde\Phi^i_{Y',Y})(\Phi_{X,X'}\ot\tilde\Phi^i_{Y,Y'})=\Phi^j_{X,X'}\ot(Y\ot Y')$$
and
$$((X\ot X')\ot\tilde\Phi^i_{Y,Y'})(\Phi_{X',X}\ot\tilde\Phi^i_{Y',Y})((X'\ot X)\ot\tilde\Phi^{-i}_{Y,Y'})=
((X\ot X')\ot\tilde\Phi^t_{Y,Y'})(\Phi^s_{X',X}\ot(Y\ot Y')).$$
The first one implies that $j=1$ and that $\tilde\Phi^i_{Y',Y}\tilde\Phi^i_{Y,Y'}=\Id_{Y\ot Y'}$ for every choice of $i$, hence $\D$ should be symmetric.
The second one reduces to: $\Phi_{X',X}\ot\tilde\Phi^i_{Y,Y'}=\Phi^s_{X',X}\ot\tilde\Phi^t_{Y,Y'}$. Then $s=1$ and the indeces $i$ and $t$ are irrelevant,
since $\D$ symmetric.
\qed\end{proof}

\begin{cor} \colabel{symmetric}
If $\C$ is a symmetric finite tensor category, then every $\C$-bimodule category $\M$ is a one-sided $\C\Del\C$-bimodule category.
\end{cor}

\bigskip

$\C$-(bi)module categories, $\C$-(bi)module functors and $\C$-(bi)module natural transformations form 2-categories. Truncating them (forgetting 
2-cells and identifying isomorphic 1-cells) we obtain categories. So, their objects are $\C$-(bi)module categories and morphisms are isomorphism 
classes of $\C$-(bi)module functors. We will denote these categories by $\C\x\Bimod, \C\x\Mod$ and $\Mod\x\C$, depending on the side of the $\C$-module structures. 

The converse of \prref{big associators}, parts 1) and 2) also holds and for a symmetric finite tensor category $\C$ we have:
\begin{equation} \eqlabel{summ one sided bimod}
\M\in\C\x\Bimod\quad \Leftrightarrow\quad \M\in\C\Del\C\x\Mod \quad\Leftrightarrow\quad \M\in\Mod\x\C\Del\C
\end{equation}
and moreover $\M$ is a one-sided $\C\Del\C$-bimodule category.

\bigskip

The construction of one-sided $\C$-bimodule categories, for a braided finite tensor category $\C$, induces in fact an embedding of categories:
$$\C\x\Mod\hookrightarrow\C\x\Bimod$$
(see \cite[Section 2.8]{DN} for the 2-category case). The module structures on the functors we will see in details in \ssref{left duals} and \ssref{right duals}.
The obtained subcategory (of one-sided $\C$-bimodule categories) we will denote by $\C^{br}\x\Mod$.

In \cite{Gr} it was proved that the 2-category of $\C$-bimodule categories is monoidal. In \cite[Proposition 4.4, Proposition 4.9]{Gr} it was proved that the associativity 
1-isomorphism $a$ strictly satisfies the pentagonal coherence axiom, as in a monoidal category, and in \cite[Proposition 3.15, Proposition 4.11]{Gr} it was proved that the coherence 
of the unity constraints with $a$ holds up to isomorphism. This means that in the category $\C^{br}\x\Mod$ the latter coherence holds strictly. Consequently, $\C^{br}\x\Mod$ 
is indeed a monoidal category with tensor product $\Del_{\C}$ and unit $\C$.

The next proof we will not do with all the rigor, the full proof we will possibly present in another paper.

\begin{prop} \prlabel{sim bimod}
For a symmetric finite tensor category $\C$ the category $(\C^{br}\x\Mod, \Del_{\C}, \C)$ is symmetric monoidal.
\end{prop}

\begin{proof}
Let $\M$ and $\N$ be two one-sided $\C$-bimodule categories. The braiding of $\C$ enables to define a $\C$-balanced functor $F: \M\Del\N\to\N\Del_{\C}\M,
F(M\Del N)=N\Del_{\C}M$. For $M\in\M, N\in\N$ and $X\in\C$ we have: $F(M\crta\ot X, N)=N\Del_{\C}(M\crta\ot X)=N\Del_{\C}(X\crta\ot M)\stackrel{\beta^{-1}_X}{\to}
(N\crta\ot X)\Del_{\C}M= (X\crta\ot N)\Del_{\C}M=F(M, X\crta\ot N)$,  where $\beta_X$ is from \equref{id C-balanced}. For the sake of this non-rigorous 
proof we will consider $b_{M,X,N}: F(M\crta\ot X, N)\to F(M, X\crta\ot N)$ to be identity.
Then the coherence condition for $F$ to be $\C$-balanced, 
$F(M, {m_l}_{X,Y,N}^{-1}) b_{M,X,Y \crta\ot N} b_{M\crta\ot X,Y,N} F({m_r}_{M,X,Y}, N) = b_{M,X \ot Y,N}$ for $M\in\M, N\in\N, X,Y\in\C$
reduces to: $({m_l}_{X,Y,N}^{-1}\Del_{\C}M)(N\Del_{\C}{m_r}_{M,X,Y})=N\Del_{\C}(M\ot (X\ot Y))$.
To check this identity we compose it with the isomorphism $((\Phi_{Y,X}^{-1}\crta\ot N)\Del_{\C}M)$. Recall that the right associator $m_r$ by
\equref{right associator} is given via ${m_r}_{M,X,Y}= {m_l}_{Y,X,M}(\Phi_{X,Y}\ot M)$. 
Now we compute:
$$\begin{array}{rl}
((\Phi_{Y,X}^{-1}\crta\ot N)\hskip-1em&\hspace{-0,2cm}\Del_{\C}M)({m_l}_{X,Y,N}^{-1}\Del_{\C}M)(N\Del_{\C}{m_l}_{Y,X,M}(\Phi_{X,Y}\ot M))\\
&= ({m_r}_{N,Y,X}^{-1}\Del_{\C}M)({m_r}_{N,Y,X}(N\crta\ot\Phi_{X,Y})\Del_{\C}M)\\
&= (\Phi_{X,Y} \crta\ot N)\Del_{\C}M.
\end{array}$$
In the first equation we applied the identity of the right associator $m_r$ for $\N$ and \equref{id C-balanced}.
In the second one we used that $\N$ is a one-sided $\C$-bimodule. 
Since $\C$ is symmetric, we may cancel out the composed factor to recover $({m_l}_{X,Y,N}^{-1}\Del_{\C}\M)(\N\Del_{\C}{m_r}_{M,X,Y})=\Id_{\M\Del_{\C}\N}$.
Thus $F$ is $\C$-balanced and it
induces a unique functor
\begin{equation} \eqlabel{flip tau}
\tau: \M\Del_{\C}\N\to\N\Del_{\C}\M, \quad M\Del_{\C} N\mapsto N\Del_{\C} M
\end{equation} 
which clearly is an isomorphism functor. It remains to see if $\tau$ is $\C$-bilinear. The isomorphism $\tau(X\crta\ot M\Del_{\C}N)\to X\crta\ot \tau(M\Del_{\C}N)$
understood as the identification $ N\Del_{\C}(X\crta\ot M)\stackrel{\beta^{-1}_X}{\to}(N\crta\ot X)\Del_{\C} M= (X\crta\ot N)\Del_{\C}M$ we will consider 
identity, similarly as we did before. 
The coherence for left $\C$-linearity of $\tau$ then reduces to $N\Del_{\C}{m_l}_{X,Y,M}={m_l}_{X,Y,N}\Del_{\C}M$.
The diagram chasing argument shows that this identity follows by the identity \equref{id C-balanced} and the rules for the right and the two-sided
associator: \equref{right associator} and \equref{mixed assoc}.
The coherence for the right $\C$-linearity similarly reduces to: ${m_r}_{N,X,Y}\Del_{\C}M=N\Del_{\C}{m_r}_{M,X,Y}$. By the relation between left and right associator
functors, this comes down to: $({m_l}_{Y,X,N}\Del_{\C}M)((\Phi_{X,Y}\crta\ot N)\Del_{\C}M) = (N\Del_{\C}{m_l}_{Y,X,M})(N\Del_{\C}(\Phi_{X,Y}\crta\ot M))$,
which is similarly proved as the coherence for the left side.
\qed\end{proof}

\section{Dual objects for bimodule categories} 

In this section we compute the left and the right dual objects for the objects of the monoidal category $(\C\x\Bimod, \Del_{\C}, \C)$, where $\C$ is a finite
tensor category, proving thus that the former is a closed monoidal category. In \ssref{Picard category} we conclude that if $\C$ is symmetric, then the left and the
right dual objects in $(\C^{br}\x\Mod, \Del_{\C}, \C)$ coincide. Actually, this is well-known in braided monoidal categories, and we saw in \prref{sim bimod}
that in order for $(\C\x\Bimod, \Del_{\C}, \C)$ (or some subcategory) to be braided, it it necessary: that $\C$ be braided, that we restrict to one-sided $\C$-bimodule
categories, and finally also that $\C$ be even symmetric.

\subsection{Inner hom objects for module categories}

There are following four pairs of adjoint functors $\C\to\C$:
\begin{equation} \eqlabel{adj}
(X\ot-, {}^*X\ot -), \quad (X^*\ot-, X\ot -), \quad (-\ot {}^*X, -\ot X), \quad (-\ot X, -\ot X^*).
\end{equation}

For a left $\C$-module category $\M$ let $\u\Hom_{\M}(M_1, M_2)$ be the inner hom object for $M_1, M_2\in\M$. It is an object in $\C$ such that
$$\Hom_{\M}(X\crta\ot M_1, M_2)\iso \Hom_{\C}(X, \u\Hom_{\M}(M_1, M_2))$$
for $X\in\C$. Then $\u\Hom_{\M}(-, M):\M\to\C^{op}$ and $\u\Hom_{\M}(M, -):\M\to\C$ are left $\C$-linear, thus they are functors in $\Fun_{\C}(\M, \C)$,
\cite[Corollary 2.10.15]{EGNO}. ($\M^{op}$ is the module category defined via relations \equref{left op} and \equref{right op}.)

\medskip

Similarly, for a right $\C$-module category $\M$ let $\o\Hom_{\M}(M_1, M_2)$ denote the corresponding inner hom object. It is such
an object in $\C$ that:
\begin{equation} \eqlabel{right inner hom}
\Hom_{\M}(M_1\crta\ot X, M_2)\iso \Hom_{\C}(X, \o\Hom_{\M}(M_1, M_2)).
\end{equation}
Then we have:
$$\begin{array}{rl}
\Hom_{\C}(Y, \hskip-1em& \hspace{-0,2cm} \o\Hom_{\M}(M_1\crta\ot X, M_2))=\Hom_{\M}((M_1\crta\ot X)\crta\ot Y , M_2)=\Hom_{\M}(M_1\crta\ot (X\ot Y ), M_2)\\
& \Hom_{\C}(X\ot Y, \o\Hom_{\M}(M_1, M_2)) \stackrel{\equref{adj}}{=}\Hom_{\C}(Y, {}^*X\ot \o\Hom_{\M}(M_1, M_2)).
\end{array}$$
This yields:
\begin{equation} \eqlabel{(3 derecha)}
\o\Hom_{\M}(M_1\crta\ot X, M_2)={}^*X\ot\o\Hom_{\M}(M_1, M_2).
\end{equation}
Similarly, we find:
$$\begin{array}{rl}
\Hom_{\C}(Y, \hskip-1em& \hspace{-0,2cm} \o\Hom_{\M}(M_1, M_2\crta\ot X))=\Hom_{\M}(M_1\crta\ot Y, M_2\crta\ot X)\stackrel{\equref{adj}}{=}\Hom_{\M}(M_1\crta\ot Y\crta\ot {}^*X, M_2)\\
& \Hom_{\C}(Y\ot {}^*X, \o\Hom_{\M}(M_1, M_2)) \stackrel{\equref{adj}}{=}\Hom_{\C}(Y, \o\Hom_{\M}(M_1, M_2)\ot X)
\end{array}$$
which yields:
\begin{equation} \eqlabel{(4) derecha}
\o\Hom_{\M}(M_1, M_2\crta\ot X)=\o\Hom_{\M}(M_1, M_2)\ot X.
\end{equation}
Then \equref{(3 derecha)} implies that $\o\Hom_{\M}(-, M):\M\to{}^{op}\C$ is right $\C$-linear and \equref{(4) derecha} implies that $\o\Hom_{\M}(M, -):\M\to\C$
is right $\C$-linear. Thus both can be seen as functors in $\Fun(\M, \C)_{\C}$. We explain below what ${}^{op}\M$ means.

\medskip

A right $\ca$-module category $\M$ gives rise to a left $\C$-module category ${}^{op}\mo$ with the action
given by \equref{left op my} and associativity isomorphisms $m^{op}_{X,Y,M}= m_{M,Y^*, X^*}$ for all $X, Y\in \ca, M\in \mo$.
Similarly, a left $\ca$-module category $\mo$ gives rise to a right $\ca$-module category ${}^{op}\mo$ with the action given via \equref{right op my}.
If $\mo$ is a $(\ca,\Do)$-bimodule category then ${}^{op}\mo$ is a
$(\Do,\ca)$-bimodule category and ${}^{op}({}^{op}\mo)\iso\M$ as $(\ca,\Do)$-bimodule categories. \vspace{-1cm}
\begin{center}
\begin{tabular}{p{4.8cm}p{2cm}p{5.8cm}}
\begin{eqnarray}  \eqlabel{left op my}
X{}^{op}\crta\ot M=M\otb X^*
\end{eqnarray}  & &
\begin{eqnarray} \eqlabel{right op my}
M{}^{op}\crta\ot X={}^*X\crta\ot M
\end{eqnarray}
\end{tabular}
\end{center} \vspace{-0,7cm}

\medskip

If $\M$ is a $(\D,\C)$-bimodule category, we find the following behaviour of the right $\C$-linear inner hom object:
$$\begin{array}{rl}
\Hom_{\C}(Y, \hskip-1em& \hspace{-0,2cm} \o\Hom_{\M}(D\crta\ot -, M))=\Hom_{\M}((D\crta\ot -)\crta\ot Y, M)=\Hom_{\M}(D\crta\ot (-\crta\ot Y), M)\\
&\stackrel{\equref{adj}}{=}\Hom_{\M}(-\crta\ot Y,  {}^*D\crta\ot M)= \Hom_{\C}(Y, \o\Hom_{\M}(-, {}^*D\crta\ot M)) .
\end{array}$$
Consequently:
\begin{equation} \eqlabel{mixed inner}
\o\Hom_{\M}(D\crta\ot -, M)=\o\Hom_{\M}(-, {}^*D\crta\ot M)
\end{equation}
for every $D\in\D$.

\subsection{Left dual object for a bimodule category} \sslabel{left duals}

For a $\C\x\D$-bimodule category $\M$ and a $\C\x\E$-bimodule category $\N$ the category
$\Fun_{\C}(\M,\N)$ is a $\D\x\E$-bimodule category via
\begin{equation} \eqlabel{left-left}
(X\crta\ot F)=F(-\crta\ot X)
\end{equation} and
\begin{equation} \eqlabel{right-left}
(F\crta\ot Y)=F(-)\crta\ot Y
\end{equation}
for $X\in\D, Y\in\E, F\in\Fun_{\C}(\M,\N)$ and $M\in\M$.

Let $\M$ be a $\C\x\D$-bimodule category and $\N$ a $\C\x\E$-bimodule category. We denote by
\begin{equation}\eqlabel{teta}
\theta_{\M,\N}: \M^{op}\Del_{\C}\N\to\Fun_{\C}(\M,\N)
\end{equation}
the $\D\x\E$-bimodule equivalence from
\cite[Proposition 3.5]{ENO} given by $M\Del_{\C}N\mapsto \u\Hom_{\M}(-, M)\crta\ot N$. 

\medskip

Denote by $R:\C\to\Fun_{\C}(\M,\M)$ the functor given by $R(X)=-\crta\ot X$, where $\M$ is a $\C$-bimodule category. It is $\C$-bilinear:
$R(C\ot X)=-\crta\ot(C\ot X)\iso(-\crta\ot C)\crta\ot X=R(X)(-\crta\ot C)\stackrel{\equref{left-left}}{=}C\crta\ot R(X)$ and
$R(X\ot C)\iso(-\crta\ot X)\crta\ot C=R(X)\crta\ot C$ for $X,C\in\C$. The coherence for the left and right $\C$-linearity of $R$ are precisely the coherence
for the right $\C$-action on $\M$.

We now define $\coev: \C\to\M^{op}\Del_{\C} \M$ such that the triangle $\langle 1\rangle$ in the following diagram commutes:
\begin{equation} \eqlabel{coev}
\scalebox{0.88}{\bfig
 \putmorphism(-50,500)(1,0)[\C`\M^{op}\Del_{\C} \M`\coev]{1170}1a
 \putmorphism(-50,0)(1,0)[\Fun_{\C}(\M,\M)`\Fun_{\C}(\M,\C)\Del_{\C}\M`\db]{1250}{-1}a
\putmorphism(-40,500)(0,-1)[\phantom{B\ot B}``R]{480}1l
\putmorphism(950,500)(0,-1)[\phantom{B\ot B}``\theta_{\M, \C}\Del_{\C}\M]{480}1r
\putmorphism(350,140)(2,1)[``\theta_{\M, \M}]{340}{-1}l
\put(40,250){\fbox{1}}
\put(700,160){\fbox{2}}
\efig}
\end{equation}
The functor $\db: \Fun_{\C}(\M,\C)\Del_{\C}\M\to\Fun_{\C}(\M,\M)$ is defined so that $\db(F\Del_{\C}M)=F(-)\crta\ot M$, for $F\in\Fun_{\C}(\M,\M), M\in\M$.
Then the triangle $\langle 2\rangle$ above commutes. It is clear that $\db$ is a $\C$-bimodule equivalence functor, and $\coev$ is a $\C$-bimodule functor.
From the definition of $\coev$ we see that $\coev(I)=\oplus_{i\in J} W_i\Del_{\C}V_i$ is such an object in $\M^{op}\Del_{\C} \M$
that $\Id_{\M}= \oplus_{i\in J} \u\Hom_{\M}(-, W_i)\crta\ot V_i$. That is, for every $M\in\M$ one has:
\begin{equation} \eqlabel{coev eq}
\oplus_{i\in J} \u\Hom_{\M}(M, W_i)\crta\ot V_i=M.
\end{equation}

\bigskip

In \cite[Corollary 3.22]{Gr} it is proved that a $\C\x\D$-bimodule category $\M$ induces an adjoint pair of functors
$\M\Del_{\D}-: \D\x\E\x\Bimod \to \C\x\E\x\Bimod: \Fun_{\C}(\M, -)$.
The unit of the adjunction evaluated at a $\D\x\E$-bimodule category $\N$
is a $\D\x\E$-bilinear functor $\alpha(\N): \N\to\Fun_{\C}(\M, \M\Del_{\D}\N)$ given by $N\mapsto -\Del_{\D}N$. If $\D=\E=\C=\N$, we get a
$\C$-bilinear functor $\alpha: \C\to\Fun_{\C}(\M, \M\Del_{\C}\C)=\Fun_{\C}(\M, \M)$ such that $\alpha(X)=-\Del_{\C}X$. The counit of the adjunction
in this case is a $\C$-bilinear functor $\beta: \M\Del_{\C}\Fun_{\C}(\M, \C)\to\C$ which is the evaluation functor.

\begin{rem} \rmlabel{invertible-equiv}
When $\M$ is an invertible $\C$-bimodule category we have that the unit of the adjunction $R: \C\to\Fun_{\C}(\M, \M)$, and
hence also the counit $\Ev:  \M\Del_{\C}\Fun_{\C}(\M, \C)\to\C$, is an equivalence \cite[Proposition 4.2]{ENO}. In this case the functor
$\coev$ from \equref{coev} is an equivalence, too.
\end{rem}

Finally, define the functor $\ev: \M\Del_{\C}\M^{op}\to\C$ through the commuting diagram:
\begin{eqnarray} \eqlabel{functor ev}
\scalebox{0.88}{\bfig
 \putmorphism(-180,500)(1,0)[\M\Del_{\C} \M^{op}`\C`\ev]{1100}1a
 \putmorphism(0,0)(1,1)[\M\Del_{\C}\Fun_{\C}(\M,\C)``]{200}0a
\putmorphism(550,110)(1,1)[``\Ev]{240}1r
\putmorphism(-40,500)(0,-1)[\phantom{B\ot B}``\M\Del_{\C}\theta_{\M, \C}]{480}1l
\efig}
\end{eqnarray}
Then $\ev$ is $\C$-bilinear and it is $\ev(M\Del_{\C}N)=\u\Hom_{\M}(M,N)$. 

\begin{prop} \prlabel{left dual}
Let $\M$ be a $\C$-bimodule category. The object $\M^{op}$ together with the functors
$$\ev: \M\Del_{\C}\M^{op}\to\C\quad\textnormal{and}\quad\coev: \C\to \M^{op}\Del_{\C}\M$$
is a left dual object for $\M$ in the monoidal category $(\C\x\Bimod, \C, \Del_{\C})$.

Consequently, $(\ev\Del_{\C} \M)(\M\Del_{\C} \coev)\iso\Id_{\M}$ and $(\M^{op}\Del_{\C} \ev)(\coev\Del_{\C} \M^{op})\iso\Id_{\M^{op}}$.

If $\M$ is an invertible $\C$-bimodule category, the functors $\ev$ and $\coev$ are $\C$-bimodule equivalence functors.
\end{prop}

\begin{proof}
Take $M\in\M$, we find: $(\ev\Del_{\C} \M)(\M\Del_{\C} \coev)(M)\iso(\ev\Del_{\C} \M)(\oplus_{i\in J} M \Del_{\C} W_i\Del_{\C}V_i)=
\oplus_{i\in J} \u\Hom_{\M}(M, W_i)\Del_{\C} V_i=M$ by \equref{coev eq}. To check the other identity, observe that
$\ev[\M\Del_{\C}(\M^{op}\Del_{\C} \ev)(\coev\Del_{\C} \M^{op})]\iso
\ev(\ev\Del_{\C} \M\Del_{\C}\M^{op})(\M\Del_{\C} \coev\Del_{\C}\M^{op})\iso\ev$, as the first axiom for dual objects is satisfied.
Now by the universal property of the evaluation functor (which is the counit of the adjunction) it follows that the second axiom also holds.

The last part follows by \rmref{invertible-equiv} and \equref{functor ev}.
\qed\end{proof}

\subsection{Right dual object for a bimodule category} \sslabel{right duals}

Let $\M$ be an $\E\x\C$-bimodule category and $\N$ a $\D\x\C$-bimodule categories. The functor category $\Fun(\M, \N)_{\C}$ is a
$\D\x\E$-bimodule category via
\begin{equation} \eqlabel{left}
(X\crta\ot F)(M)=X\crta\ot F(M)
\end{equation}
\begin{equation} \eqlabel{rright}
(F\crta\ot Y)(M)=F(Y\crta\ot M)
\end{equation}
for $X\in\D, Y\in\E, F\in\Fun(\M, \N)_{\C}$ and $M\in\M$.

\begin{lma} \lelabel{sigma functor}
Let $\C, \D, \E$ be finite tensor categories, let $\M$ be an $\E\x\C$-bimodule category and $\N$ a $\D\x\C$-bimodule category. The functor
\begin{equation}\eqlabel{sigma}
\sigma_{\M,\N}: \M\Del_{\C}{}^{op}\N\to\Fun(\N,\M)_{\C}
\end{equation}
given by $\sigma(M\Del_{\C}N)=M\crta\ot \o\Hom_{\N}(-, N)$ is an $\E\x\D$-bimodule equivalence.
\end{lma}

\begin{proof}
Observe that the functor $\sigma_1: \M\Del{}^{op}\N\to\Fun(\N,\M)$ given by $\sigma_1(M\Del N)=M\crta\ot \o\Hom_{\N}(-, N)$
is $\C$-balanced:
$$\begin{array}{rl}
\sigma_1(M\crta\ot X, N)\hskip-1em&\hspace{-0,2cm} =(M\crta\ot X)\crta\ot \o\Hom_{\N}(-, N)\iso M\crta\ot (X\ot\o\Hom_{\N}(-, N))\\
&=M\crta\ot (\o\Hom_{\N}(-, N)\ot X^*) \stackrel{\equref{(4) derecha}}{=}M\crta\ot \o\Hom_{\N}(-, N\crta\ot X^*) \\
& \stackrel{\equref{right op my}}{=} M\crta\ot \o\Hom_{\N}(-, X{}^{op}\crta\ot N)=\sigma_1(M, X{}^{op}\crta\ot N).
\end{array}.$$
The isomorphism $\sigma_1(M\crta\ot X, N)\to\sigma_1(M, X\crta\ot N)$
is thus given by the composition of the identity \equref{(4) derecha} and the right associativity functor $m_r$ for $\M$; it satisfies the coherence
for a $\C$-balanced functor because of the coherence for the right $\C$-action on $\M$ and the coherence of the successive application of
\equref{(4) derecha} for objects $X, Y\in\C$. Thus $\sigma_1$ induces $\sigma_2: \M\Del_{\C}{}^{op}\N\to\Fun(\N,\M)$.
The functor $\sigma_2(M\Del_{\C}N)=M\crta\ot \o\Hom_{\N}(-, N)$ is right $\C$-linear, as so is $\o\Hom_{\N}(-, N): \N\to {}^{op}\C$.
Thus $\sigma_{\M,\N}$ is well-defined. It is clearly left $\C$-linear. For right $\D$-linearity we find:
$$\begin{array}{rl}
\sigma_{\M,\N}(M\Del_{\C} \hskip-1em&\hspace{-0,2cm} N{}^{op}\crta\ot D)=M\crta\ot\o\Hom_{\N}(-, N{}^{op}\crta\ot D)\stackrel{\equref{right op my}}{=}
M\crta\ot\o\Hom_{\N}(-, {}^*D\crta\ot N)\\
& \stackrel{\equref{mixed inner}}{=}M\crta\ot\o\Hom_{\N}(D\crta\ot -, N)\stackrel{\equref{rright}}{=}M\crta\ot\o\Hom_{\N}( -, N)\crta\ot D=
\sigma_{\M,\N}(M\Del_{\C}N)\crta\ot D.
\end{array}$$

The proof that $\sigma_{\M,\N}$ is an equivalence is analogous to the proof of \cite[Theorem 3.20, Lemma 3.21]{Gr}. The inverse of $\sigma_{\M,\N}$
is induced by the functor $J: \Fun(\N,\M)_{\C}\to \M\Del_{\C}{}^{op}\N$ such that $J(F)$ is the representing object of the functor $M\Del_{\C} N\mapsto\Hom(M, F(N))$,
for $F\in\Fun(\N,\M)_{\C}$. That is, there is an equivalence $\Hom_{\M\Del_{\C}{}^{op}\N}(M\Del_{\C} N, J(F))=\Hom_{\M}(M, F(N))$.
\qed\end{proof}

(In \cite[Corollary 3.4.11]{Doug} the category equivalence from above lemma appears with a different, non-explicit proof.)

\medskip

Let $L:\C\to\Fun(\M,\M)_{\C}$ be the functor given by $L(X)=X\crta\ot -$. Then $L$ is $\C$-bilinear. For left $\C$-linearity we have:
$L(X\ot C)=(X\ot C)\crta\ot -\iso X\crta\ot (C\crta\ot -)=X\crta\ot L(C)$ for $X,C\in\C$, and for the right one:
$L(C\ot X)\iso C\crta\ot (X\crta\ot -)\stackrel{\equref{rright}}{=}L(C)\crta\ot X$.

\bigskip

Now let $\M$ be a $\C$-bimodule category.
From \leref{sigma functor} we have that $\sigma:=\sigma_{\C,\M}: {}^{op}\M\stackrel{\simeq}{\to}\Fun(\M,\C)_{\C}$ as
$\C$-bimodule categories. 
We define $\crta\coev: \C\to\M\Del_{\C} {}^{op}\M$ such that the triangle $\langle 1\rangle$ in the following diagram commutes:
\begin{equation} \eqlabel{crta coev}
\scalebox{0.88}{\bfig
 \putmorphism(-50,500)(1,0)[\C`\M\Del_{\C} {}^{op}\M`\crta\coev]{1070}1a
 \putmorphism(-50,0)(1,0)[\Fun(\M,\M)_{\C}`\M\Del_{\C}\Fun(\M,\C)_{\C}`\crta\db]{1250}{-1}a
\putmorphism(-40,500)(0,-1)[\phantom{B\ot B}``L]{480}1l
\putmorphism(950,500)(0,-1)[\phantom{B\ot B}``\M\Del_{\C}\sigma]{480}1r
\putmorphism(350,140)(2,1)[``\sigma_{\M, \M}]{340}{-1}l
\put(40,250){\fbox{1}}
\put(700,160){\fbox{2}}
\efig}
\end{equation}
Then $\crta\coev$ too is a $\C$-bimodule functor. The functor $\crta\db: \M\Del_{\C}\Fun(\M,\C)_{\C}\to\Fun(\M,\M)_{\C}$ is defined so that
$\crta\db(M\Del_{\C}F)=M\crta\ot F(-)$, for $F\in\Fun(\M,\M)_{\C}, M\in\M$.
Then the triangle $\langle 2\rangle$ above commutes, and $\crta\db$ is a $\C$-bimodule equivalence functor.
From the definition of $\crta\coev$ we see that $\crta\coev(I)=\oplus_{i\in J} V_i\Del_{\C}W_i$ is such an object in $\M\Del_{\C} {}^{op}\M$
that $\Id_{\M}= \oplus_{i\in J} V_i\crta\ot\o\Hom_{\M}(-, W_i)$. That is, for every $M\in\M$ one has:
\begin{equation} \eqlabel{crta coev eq}
\oplus_{i\in J} V_i\crta\ot\o\Hom_{\M}(M, W_i)=M.
\end{equation}

\bigskip

The next properties are easy to deduce:

\begin{lma}
\begin{itemize}
\item[1)] ${}^{op}(\Fun(\M, \N)_{\C})\simeq\Fun(\N, \M)_{\C}$ as $\D\x\E$-bimodule categories, for ${}_{\D}\M_{\C}, {}_{\E}\N_{\C}$;
\item[2)] ${}^{op}(\M\Del_{\C}\N)\simeq {}^{op}\N\Del_{\C}{}^{op}\M$ as $\E\x\D$-bimodule categories, for ${}_{\D}\M_{\C}, {}_{\C}\N_{\E}$.
\end{itemize}
\end{lma}

The following proposition can be proved directly, or alternatively, by using successive applications of the equivalence $\sigma$ from
\equref{sigma} and the properties from the above lemma:

\begin{prop} \prlabel{Fun adj}
Let $\N$ be a $\D\x\C$-bimodule category, $\M$ be a $\C\x\E$-bimodule category and $\A$ an $\F\x\E$-bimodule category. Then there is an
isomorphism of $\F\x\D$-bimodule categories:
$$\Fun(\N\Del_{\C}\M, \A)_{\E}\iso\Fun(\N, \Fun(\M, \A)_{\E})_{\C}.$$
\end{prop}

From \prref{Fun adj} we see that given a $\C\x\E$-bimodule category $\M$ there is an adjoint pair of functors:
$-\Del_{\C}\M: \D\x\C\x\Bimod\to\D\x\E\x\Bimod: \Fun(\M, -)_{\E}$. The unit of the adjunction evaluated at a $\C\x\D$-bimodule category $\N$
is a $\C\x\D$-bilinear functor $\alpha(\N): \N\to\Fun(\M, \N\Del_{\C}\M)_{\E}$ given by $N\mapsto N\Del_{\C}-$. If $\D=\C=\N$, we get a
$\C$-bilinear functor $\alpha: \C\to\Fun(\M, \C\Del_{\C}\M)_{\E}=\Fun(\M, \M)_{\E}$ such that $\alpha(X)=X\Del_{\C}-$. The counit of the adjunction
in this case is a $\C$-bilinear functor $\beta: \Fun(\M, \C)_{\E}\Del_{\C}\M\to\C$ which is the evaluation functor.

\begin{rem} \rmlabel{invertible-equiv right}
When $\E=\C$ and $\M$ is an invertible $\C$-bimodule we have that the unit of the adjunction $L: \C\to\Fun(\M, \C\Del_{\C}\M)_{\C}$, and
hence also the counit $\crta\Ev: \Fun(\M, \C)_{\C}\Del_{\C}\M\to\C$, is an equivalence \cite[Proposition 4.2]{ENO}. In this case the functor
$\crta\coev$ from \equref{crta coev} is an equivalence, too.
\end{rem}

For a $\C$-bimodule category $\M$ we define the functor $\crta\ev: {}^{op}\M\Del_{\C}\M\to\C$ through the commuting diagram:
\begin{eqnarray} \eqlabel{functor crta ev}
\scalebox{0.88}{\bfig
 \putmorphism(-180,500)(1,0)[{}^{op}\M\Del_{\C} \M`\C`\crta\ev]{800}1a
 \putmorphism(-280,0)(1,1)[\Fun(\M,\C)_{\C}\Del_{\C}\M``]{200}0a
\putmorphism(240,130)(1,1)[``\crta\Ev]{240}1r
\putmorphism(-40,500)(0,-1)[\phantom{B\ot B}``\sigma\Del_{\C}\M]{480}1l
\efig}
\end{eqnarray}
Then $\crta\ev$ is $\C$-bilinear and it is $\crta\ev(M\Del_{\C}N)=\o\Hom_{\M}(N,M)$.

\begin{prop} \prlabel{ev-db-equiv}
Let $\C$ be a braided finite tensor category and let $\M$ be a $\C$-bimodule category. The object ${}^{op}\M$ together with the functors
$$\crta\ev: {}^{op}\M\Del_{\C}\M\to\C\quad\textnormal{and}\quad\crta\coev: \C\to \M\Del_{\C}{}^{op}\M$$
is a right dual object for $\M$ in the monoidal category $(\C\x\Bimod, \C, \Del_{\C})$.

Consequently, $(\M\Del_{\C} \crta\ev)(\crta\coev\Del_{\C} \M)\iso\Id_{\M}$ and $(\crta\ev\Del_{\C} {}^{op}\M)({}^{op}\M\Del_{\C} \crta\coev)\iso\Id_{{}^{op}\M}$.

If $\M$ is an invertible $\C$-bimodule category, the functors $\crta\ev$ and $\crta\coev$ are $\C$-bimodule equivalence functors.
\end{prop}

\begin{proof}
Take $M\in\M$, we find: $(\M\Del_{\C} \crta\ev)(\crta\coev\Del_{\C} \M)(M)=(\M\Del_{\C} \crta\ev)(\oplus_{i\in J} V_i\Del_{\C}W_i \Del_{\C} M)=
\oplus_{i\in J} V_i\Del_{\C}\o\Hom_{\M}(M, W_i)=M$ by \equref{crta coev eq}. To check the other identity, observe that
$\crta\ev[(\crta\ev\Del_{\C} {}^{op}\M)({}^{op}\M\Del_{\C} \crta\coev) \Del_{\C}\M]=
\crta\ev(\crta\ev\Del_{\C} {}^{op}\M\Del_{\C}\M)({}^{op}\M\Del_{\C} \crta\coev\Del_{\C}\M)=
\crta\ev({}^{op}\M\Del_{\C}\M\Del_{\C} \crta\ev)({}^{op}\M\Del_{\C} \crta\coev\Del_{\C}\M)=\crta\ev$, as the first axiom for dual objects is satisfied.
Now by the universal property of the evaluation functor (which is the counit of the adjunction) it follows that the second axiom also holds.

The last part follows by \rmref{invertible-equiv right} and \equref{functor crta ev}.
\qed\end{proof}

\begin{rem}
In \cite[Proposition 4.22]{Schauma} the left $\M^{\sharp}$ and the right dual ${}^{\sharp}\M$ of a $\C$-bimodule category are constructed. They are versions
of our $\M^{op}$ and ${}^{op}\M$.
\end{rem}

\subsection{Some applications}

We will prove here some claims that will be useful in our future computations. 

\begin{lma}\lelabel{basic lema}
Let 
$\F, \G: \M\to\N$ be $\D\x\C$-bimodule functors and let $\Pp$ be an invertible $\C$-bimodule category.
\begin{enumerate}
\item It is $\F\Del_{\C}\Id_{\Pp}\simeq\G\Del_{\C}\Id_{\Pp}$ if and only if $\F\simeq\G$.
\item If $\HH: \Pp\to\Ll$ is a $\C$-bimodule equivalence functor, it is $\F\Del_{\C}\HH\simeq\G\Del_{\C}\HH$ if and only if $\F\simeq\G$.
\end{enumerate}
\end{lma}

\begin{proof}
\begin{enumerate}
\item
We tensor the identity $\F\Del_{\C}\Id_{\Pp}\simeq\G\Del_{\C}\Id_{\Pp}$ from the right by $\Id_{\Pp^{op}}$ and then ``conjugate'' by 
$\Id_{\N}\Del_{\C} \ev_{\Pp}$. We get:
$(\Id_{\N}\Del_{\C} \ev_{\Pp})(\F\Del_{\C}\Id_{\Pp}\Del_{\C}\Id_{\Pp^{op}})(\Id_{\M}\Del_{\C} \ev_{\Pp}^{-1})\simeq
(\Id_{\N}\Del_{\C} \ev_{\Pp})(\G\Del_{\C}\Id_{\Pp^{op}}\Del_{\C}\Id_{\Pp})(\Id_{\M}\Del_{\C} \ev_{\Pp}^{-1})$, which is the same as saying that  
$\F\simeq\F\Del_{\C}\Id_{\C}=\G\Del_{\C}\Id_{\C}\simeq\G$.
\item We compose the identity $\F\Del_{\C}\HH\simeq\G\Del_{\C}\HH$ from the left by $\Id_{\N}\Del_{\C}\HH^{-1}$ and apply the part 1).
\end{enumerate}
\qed\end{proof}

For a $\C\x\D$-bimodule functor $\F:\M\to\N$ the $\D\x\C$-bimodule functor ${}^{op}\F: {}^{op}\N\to{}^{op}\M$ is given by ${}^{op}\F=\sigma_{\M, \C}^{-1}\comp\F^*\comp\sigma_{\N, \C}$.
Here $\F^*: \Fun_{\C}(\N, \C)\to\Fun_{\C}(\M, \C)$ is given by $\F^*(G)=G\comp\F$ and $\sigma$ is the equivalence from \equref{sigma}.
If $\F$ is an equivalence, it is $(\F^*)^{-1}=(\F^{-1})^*=-\comp\F^{-1}$ and consequently: $({}^{op}\F)^{-1}={}^{op}(\F^{-1})$.

\begin{lma} \lelabel{alfa dagger}
Let $\C$ be a finite tensor category and $\M$ an exact $\C$-bimodule category and let $\alpha: \M\to\C$ be a $\C$-bimodule equivalence. Then:
$\crta\ev\simeq({}^{op}\alpha)^{-1} \Del_{\C}\alpha$ as right $\C$-linear functors.
\end{lma}

\begin{proof}
Take $M\in {}^{op}\M$ and $N\in\M$. We have that $\sigma_{\C, \M}(M)=\o\Hom_{\M}(-, M)$ is a right $\C$-linear functor $\M\to \C$.
Observe that $({}^{op}\alpha)^{-1}(M)=\sigma_{\C, \C}^{-1}\comp(\alpha^{-1})^*\comp\sigma_{\C, \M}(M)=
\sigma_{\C, \C}^{-1}(\o\Hom_{\M}(\alpha^{-1}(-), M))$. Then:
\begin{eqnarray*}
(({}^{op}\alpha)^{-1} \Del_{\C}\alpha)(M\Del_{\C}N)&=&
({}^{op}\alpha)^{-1} (M)\Del_{\C} \alpha(N)\\
&=& \sigma_{\C, \C}^{-1}(\o\Hom_{\M}(\alpha^{-1}(-), M)) \Del_{\C} \alpha(N)\\
&=& \sigma_{\C, \C}^{-1}(\o\Hom_{\M}(\alpha^{-1}(-), M)\crta\ot \alpha(N)) \\
&\stackrel{\equref{rright}}{=}& \sigma_{\C, \C}^{-1}(\o\Hom_{\M}(\alpha^{-1}(\alpha(N)\ot -), M))\\
&=& \sigma_{\C, \C}^{-1}(\o\Hom_{\M}(N\crta\ot -, M)).
\end{eqnarray*}
The third equality holds since $\sigma$ is right $\C$-linear and in the last equality we used the fact that
$\alpha^{-1}$ is right $\C$-linear: 
$\o\Hom_{\M}(\alpha^{-1}(\alpha(N)\ot X), M)=\o\Hom_{\M}(\alpha^{-1}(\alpha(N))\crta\ot X, M)=\o\Hom_{\M}(N\crta\ot X, M)$.
On the other hand, we have: $\crta\ev(M\Del_{\C}N)=\o\Hom_{\M}(N,M)$ and $\sigma_{\C, \C}(\crta\ev(M\Del_{\C}N) )=\o\Hom_{\C}(-,\o\Hom_{\M}(N,M) )$.
We should prove that $\o\Hom_{\M}(N\crta\ot -, M)\iso\o\Hom_{\C}(-,\o\Hom_{\M}(N,M) )$. For arbitrary $X, Y\in\C$ we find:
$$\begin{array}{rl}
\Hom_{\C}(X, \o\Hom_{\C}(Y, \hskip-1em&\hspace{-0,2cm} \o\Hom_{\M}(N,M)) =\Hom_{\C}(Y\ot X, \o\Hom_{\M}(N,M))\\
&=\Hom_{\M}(N\crta\ot (Y\ot X), M)=\Hom_{\M}((N\crta\ot Y)\crta\ot X, M)\\
&=\Hom_{\C}(X, \o\Hom_{\M}(N\crta\ot Y, M))
\end{array}$$
and the first claim follows. Note that $\crta\ev: {}^{op}\M\Del_{\C}\M\to\C$, while $({}^{op}\alpha)^{-1} \Del_{\C}\alpha: {}^{op}\M\Del_{\C}\M\to\C^{op}\Del_{\C}\C$,
this is why the equality is of right $\C$-module functors.
\qed\end{proof}

\subsection{Dual objects for one-sided bimodules over symmetric finite tensor categories} \sslabel{Picard category}

Let $\C$ be a braided finite tensor category. Exact invertible one-sided $\C$-bimodule categories, $\C$-bimodule equivalences and $\C$-bimodule natural 
isomorphisms form a 2-category and indeed a 2-groupoid. They were studied in \cite[Section 4.4]{ENO}, \cite[Section 2.8]{DN}. However, similarly as in \seref{braided}, 
truncating this 2-category we obtain a category (indeed a 2-group) $\dul\Pic(\C)$. Its objects are exact invertible one-sided $\C$-bimodule categories 
and morphisms are isomorphism classes of $\C$-bimodule equivalences. 
From \prref{sim bimod} it follows that when $\C$ is symmetric, so is $\dul\Pic(\C)$ (as a monoidal subcategory of $(\C^{br}\x\Mod, \Del_{\C}, \C)$).

For one-sided $\C$-bimodules we have: $\Fun_{\C}(\M, \C)=\Fun(\M, \C)_{\C}$, their $\C$-module structures \equref{left}, \equref{rright}, \equref{left-left}
and \equref{right-left} coincide. Then ${}^{op}\M=\M^{op}$. We also have that the functors $\u\Hom_{\M}(M, -)$ and $\o\Hom_{\M}(M, -)$ coincide. (Namely,
they are right adjoint functors to $-\crta\ot M$ and $M\crta\ot -$ from $\C$ to $\M$, respectively, for every $M\in\M$,
and the latter are equal as functors in $\dul\Pic(\C)$. Then $\u\Hom_{\M}(M, -)$ and $\o\Hom_{\M}(M, -)$ are equal for every $M\in\M$,
hence the bifunctors $\u\Hom_{\M}(-, -)$ and $\o\Hom_{\M}(-, -)$ are equal, in particular so are $\u\Hom_{\M}(-, M)$ and $\o\Hom_{\M}(-,M)$.)

If $\C$ is symmetric, we then have: $\sigma\comp\tau=\theta$, where $\tau$ is from \equref{flip tau}, and also: $\ev\comp\tau=\crta\ev$ and
$\tau\comp\coev=\crta\coev$. Consequently, the left and the
right dual object are isomorphic in this setting.

\medskip

Recall from \equref{functor ev} that $\ev: \M\Del_{\C}\M^{op}\to\C$ is given by $\ev(M\Del_{\C}N)=\u\Hom_{\M}(M,N)$ for $M\Del_{\C}N\in\M\Del_{\C}\N$
and set $\coev(I)=\oplus_{i\in J} W_i\Del_{\C}V_i\in\M^{op}\Del_{\C} \M$. 
Then $\ev\comp\tau\comp\coev(I)=\oplus_{i\in J}\u\Hom_{\M}(V_i, W_i)$.

\begin{cor} 
For a symmetric finite tensor category $\C$ and $\M\in\dul{\Pic}(\C)$, it is $\ev^{-1}
\simeq\coev$. In particular, it is $\oplus_{i\in J}\u\Hom_{\M}(V_i, W_i)\iso I$.
\end{cor}

\begin{proof}
By \prref{left dual} the functors $\ev$ and $\coev$ are equivalences. The identity $(\ev\Del_{\C}\M)(\M\Del_{\C}\coev)\simeq\id_{\M}$
yields $\M\Del_{\C}\coev\simeq\ev^{-1}\Del_{\C}\M$. Compose this with the isomorphism $\tau$ 
and apply \leref{basic lema}, 1) to get the claim.
\qed\end{proof}

\section{Amitsur cohomology over symmetric finite tensor categories} \selabel{Amitsur coh}

Amitsur cohomology was first introduced in \cite{Am} for commutative algebras over fields, it can be viewed as an
affine version of \v Cech cohomology. It was further developed in \cite{CR,KO1}. For more details see \cite{Cae}.
We construct an analogous cohomology for symmetric finite tensor categories.
\bigskip


Given a tensor functor $\eta: \C\to\E$, then $\E$ is a left (and similarly a right) $\C$-module category.
The action bifunctor $\C\times\E\to\E$ is given by $\ot(\eta\times\Id_{\E})$, where $\ot$ is the tensor product in $\E$, and the associator functor is
\begin{equation} \eqlabel{associator}
m_{X,Y,F}=\alpha_{\eta(X), \eta(Y), E}(\xi_{X,Y}\ot E)
\end{equation}
for every $X, Y\in\C$ and $E\in\E$, where $\xi_{X,Y}: \eta(X\ot Y)\to\eta(X)\ot\eta(Y)$
determines the monoidal structure of the functor $\eta$ and $\alpha$ is the associativity constraint for $\E$. The constraint for the action of the unit
is defined in the obvious manner. Moreover, $\E$ is a $\C\x\E$-bimodule category with the bimodule constraint
$\gamma_{X, E, F}: (X\crta\ot E)\ot F\to X\crta\ot (E\ot F)$ for $X\in\C, E,F\in\E$, given via $\gamma_{X, E, F}=\alpha_{\eta(X), E,F}$.

\medskip

\begin{defnlma}
Let $\F: \D\to\C$ be a tensor functor.
\begin{itemize}
\item
Given another tensor functor $\G: \D'\to\C$ and a $\C$-bimodule category $\M$, the category ${}_{\F}\M_{\G}$ equal to $\M$ as an abelian category with actions:
$$X\crta\ot M\crta\ot X'=\F(X)\crta\ot M\crta\ot \G(X')$$
for all $X\in\D, X'\in\D'$ and $M\in\M$ is a $\D\x\D'$-bimodule category.
\item
If $\HH: \C\to\E$ is another tensor functor and $\M$ is a left $\E$-module category, then there is an obvious equivalence of left $\D$-modules categories:
\begin{equation}\eqlabel{functor change}
{}_{\F}\C\Del_{\C}{}_{\HH}\M\simeq {}_{\HH\F}\M.
\end{equation}
\end{itemize}
\end{defnlma}

\begin{proof}
For the second part, we have that the functor $\beta: {}_{\F}\C\Del_{\C}{}_{\HH}\M\to {}_{\HH\F}\M$, given by $C\Del_{\C}M\mapsto \HH(C)\crta\ot M$,
is $\C$-balanced with the natural isomorphism $b_{C,X,M}=m_{C,X,M}=\alpha_{\HH(C), \HH(X), M}(\xi_{X,Y}\crta\ot M)$,
similarly as in \equref{associator}. The coherence \equref{C-balanced} for this balance is precisely the coherence for $\M$ to be
a left $\C$-module category (through $\HH$).
\qed\end{proof}

\medskip

Fix $\C$ a finite tensor category. In order to simplify the notation we will often write: $X_1\cdots X_n$ for the tensor product $X_1\ot \dots\ot X_n$ in $\C$.
An object $X\in \C^{\Del n}$ we will write as $X=X^1\Del X^2\Del \cdots \Del X^n$, where the direct summation is understood implicitly.
Recall that the category $\C^{\Del n}$ is a finite tensor category with unit object $I^{\Del n}$ and the componentwise tensor product: $(X^1\Del X^2\Del \cdots \Del X^n)\odot
(Y^1\Del Y^2\Del \cdots \Del Y^n)=X_1Y_1\Del\cdots X_nY_n$.

\medskip

We consider the functors $e^n_i: \C^{\Del n}\to\C^{\Del (n+1)}$ with $i=1,\cdots, n+1$, given by
\begin{equation} \eqlabel{e's}
e_i^n(X^1\Del\cdots\Del X^n) = X^1\Del\cdots \Del I\Del X^i\Del\cdots\Del X^n
\end{equation}
They are clearly tensor functors. We have:

\begin{lma} \lelabel{e_i's rel}
For $i\geq j\in \{1,\cdots,n+1\}$ it is
\begin{equation}\eqlabel{e functors}
e_j^{n+1}\circ e_i^n=e_{i+1}^{n+1}\circ e_j^n.
\end{equation}
\end{lma}

\bigskip

Let $P$ be an additive covariant functor from a full subcategory of the category of (symmetric) tensor categories that contains all Deligne tensor
powers $\C^{\Del n}$ of $\C$ to abelian groups. We define $\C^{\Del 0}=k$. Then we consider
$$\delta_n=\sum_{i=1}^{n+1} (-1)^{i-1}P(e^n_i):\ P(\C^{\Del n})\to P(\C^{\Del (n+1)}).$$
It is straightforward to show, using \leref{e_i's rel}, that $\delta_{n+1}\circ \delta_n=0$, so we obtain a complex:
$$ \bfig \putmorphism(0, 0)(1, 0)[0`P(\C)`]{420}1a
\putmorphism(400, 0)(1, 0)[\phantom{P(S)}`P(\C^{\Del 2})`\delta_1]{560}1a
\putmorphism(960, 0)(1, 0)[\phantom{P(S^{\ot 2})}` P(\C^{\Del 3})`\delta_2]{600}1a
\putmorphism(1550, 0)(1, 0)[\phantom{P(S^{\ot 3})}`\cdots`\delta_3]{550}1a
\efig
$$
We will call it {\em Amitsur complex $C(\C/vec, P)$}. The notation $\C/vec$ expresses that the relative Deligne products of the copies of 
the category $\C$ are taken over the category $vec$. In all the cohomology groups that we will define below we will adopt the notation: 
$H^n(\C, F)=H^n(\C/vec, F)$ for a suitable functor $F$ and every $n=0,1,\dots$. We have:
$$Z^n(\C, P)=\Ker\delta_n,~~B^n(\C, P)=\im\delta_{n-1}~~~\textnormal{and}~~~H^n(\C, P)=Z^n(\C, P)/B^n(\C, P).$$
We will call $H^n(\C, P)$ the $n$-th Amitsur cohomology group of $\C$ with values in $P$. Elements in $Z^n(\C,P)$ are called $n$-cocycles, and elements
in $B^n(\C,P)$ are called $n$-coboundaries.

\medskip

We will consider the cases: $P=\Pic$, where $\Pic(\C)$ is the Picard group of a braided monoidal category $\C$, consisting of equivalence classes of exact invertible
one-sided $\C$-bimodule categories (\cite[Section 4.4]{ENO}, \cite[Section 2.8]{DN}), and $P=\Inv$, where $\Inv(\C)$ is the group of isomorphism classes 
of invertible objects of $\C$, which we define below. Since we need $P$ to map to abelian groups, we need to work only with symmetric categories $\C$.

\medskip

An object $X\in\C$ is called
{\em invertible} if there exists an object $Y\in\C$ such that $X\ot Y\iso I\iso Y\ot X$. For such an object $Y$ we will write $Y=X^{-1}$.
Isomorphism classes of invertible objects in $\C$ form a group $\Inv(\C)$ with the product induced by the tensor product in $\C$. If
$\C$ is braided $\Inv(\C)$ is an abelian group (in two ways). An isomorphism class of an object $X$ we will denote by $\crta X$. 

\medskip

\begin{lma} \lelabel{Del braided}
For a braided (symmetric) finite tensor category $\C$ the category $\C^{\Del n}$ is a braided (symmetric) finite tensor category for every $n\in\mathbb N$.
\end{lma}

If $\C$ is braided, every left $\C^{\Del n}$-module category is a $\C^{\Del n}$-bimodule category so we may consider the Picard group $\Pic(\C^{\Del n})$.
If $\C$ is symmetric by \prref{sim bimod} the monoidal category $((\C^{\Del n})^{br}\x\Bimod, \Del_{\C^{\Del n}}, \C^{\Del n})$ is symmetric and we have:

\begin{cor} \colabel{Pic abelian}
For a symmetric finite tensor category $\C$ the Picard group $\Pic(\C^{\Del n})$ is abelian for every $n\in\mathbb N$.
\end{cor}

\subsection{The Picard category of a symmetric finite tensor category and the Amitsur cohomology} \sslabel{deltaPic}

{\em Throughout this subsection let $\C$ be a symmetric finite tensor category}. Then $\C^{\Del n}$ is a symmetric finite tensor category and the category $\dul\Pic(\C^{\Del n})$
is symmetric monoidal with the tensor product $\Del_{\C^{\Del n}}$. The objects of $\dul\Pic(\C^{\Del n})$ are exact invertible one-sided $\C^{\Del n}$-bimodule
categories $\M$ so that there are equivalence functors $\M\Del_{\C^{\Del n}}\M^{op}\simeq \C^{\Del n}$ (and $\M^{op}\Del_{\C^{\Del n}}\M\simeq\C^{\Del n}$).

\medskip

Let us consider the functors $E^n_i: \dul\Pic(\C^{\Del n})\to\dul\Pic(\C^{\Del (n+1)})$ for $i=1,\cdots,n+1$ given by
\begin{equation} \eqlabel{E_i functors}
E^n_i(\M)=\M_i=\M\Del_{\C^{\Del n}} {}_{e^n_i}\C^{\Del (n+1)}
\end{equation}
and
$$E^n_i(F)=F_i=F\Del_{\C^{\Del n}} {}_{e^n_i}\C^{\Del (n+1)}$$
for every object $\M$ and every functor $F$ in $\dul\Pic(\C^{\Del n})$, with $e^n_i$'s from \equref{e's}.

\begin{lma} \lelabel{M_ij}
For $i\geq j\in \{1,\dots,n+1\}$ and $\M\in \dul{\Pic}(\C^{\Del n})$, we have a natural isomorphism:
\begin{equation}\eqlabel{2.2.1.1b}
\M_{ij}\cong \M_{j(i+1)}.
\end{equation}
Here $\M_{ij}=E_j^{n+1}\comp E_i^n(\M)$.
\end{lma}

\begin{proof}
\begin{eqnarray*}
\M_{ij}&=&
(\M\Del_{\C^{\Del n}}{}_{e_i^n}\C^{\Del (n+1)})\Del_{\C^{\Del (n+1)}}{}_{e_j^{n+1}}{\C^{\Del (n+2)}}\\
&\congo{\equref{functor change}}&
\M\Del_{\C^{\Del n}}{}_{(e_j^{n+1}\circ e_i^n)}{\C^{\Del (n+2)}}
\equal{\equref{e functors}}
\M\Del_{\C^{\Del n}}{}_{(e_{i+1}^{n+1}\circ e_j^n)}{\C^{\Del (n+2)}}\\
&\congo{\equref{functor change}}&
(\M\Del_{\C^{\Del n}}{}_{e_j^n}\C^{\Del (n+1)})\Del_{\C^{\Del (n+1)}}{}_{e_{i+1}^{n+1}}{\C^{\Del (n+2)}}
=\M_{j(i+1)}.
\end{eqnarray*}
\qed\end{proof}


Now for every non-zero $n\in\mathbb N$, we define a functor
\begin{equation} \eqlabel{delta_n Pic}
\delta_{n}:\ \dul{\Pic}(\C^{\Del n})\to \dul{\Pic}(\C^{\Del (n+1)}),
\end{equation}
by
$$\delta_{n}(\M)=\M_1\Del_{\C^{\Del (n+1)}}\M^{op}_2\Del_{\C^{\Del (n+1)}}\cdots
\Del_{\C^{\Del (n+1)}}\N_{n+1},$$
$$\delta_{n}(F)=F_1\Del_{\C^{\Del (n+1)}} (F^{op}_2)^{-1}\Del_{\C^{\Del (n+1)}}\cdots \Del_{\C^{\Del (n+1)}} (G_{n+1})^{\pm 1},$$
with $\N=\M$ or $\M^{op}$ and $G=F$ or $F^{op}$ depending on whether $n$ is even or odd.
Up to the permutation of the factors in the relative Deligne tensor product - we use the fact that $\dul{\Pic}(\C^{\Del n})$ is symmetric - it is clear that the functor $\delta_n$ is monoidal.

\begin{rem}
Throughout we will use similar switch functor isomorphisms in the identities.
\end{rem}

Computations similar to the computations in the proof of the previous lemma show that:
\begin{equation}\eqlabel{double delta M}
\delta_{n+1}\delta_{n}(\M)\simeq(\boxtimes_{\C^{\Del (n+2)}})_{j=2}^{n+2}(\boxtimes_{\C^{\Del (n+2)}})_{i=1}^{j-1} \hspace{0,2cm} (\M_{ij}\Del_{\C^{\Del (n+2)}} \M_{ij}^{op}),
\end{equation}
\begin{equation}\eqlabel{double delta F}
\delta_{n+1}\delta_{n}(F)\simeq(\boxtimes_{\C^{\Del (n+2)}})_{j=2}^{n+2}(\boxtimes_{\C^{\Del (n+2)}})_{i=1}^{j-1} \hspace{0,2cm} (F_{ij}\Del_{\C^{\Del (n+2)}} (F_{ij}^{op})^{-1}),
\end{equation}
so we have a natural equivalence:
$$\lambda_{\M}= (\boxtimes_{\C^{\Del (n+2)}})_{j=2}^{n+2}(\boxtimes_{\C^{\Del (n+2)}})_{i=1}^{j-1} \ev_{\M_{ij}}:\ \delta_{n+1}\delta_n(\M)\to \C^{\Del (n+2)}.$$
(Here $\lambda_{\M}$ is an equivalence by \prref{ev-db-equiv}.) Similarly, one gets:
\begin{equation}\eqlabel{delta-lambda}
\delta_{n+2}(\lambda_{\M})\simeq\lambda_{\delta_n(\M)}.
\end{equation}

\begin{rem} \rmlabel{lambda mult}
For $\M, \N\in\dul\Pic(\C^{\Del n})$ it is: $\lambda_{\M\Del_{\C^{\Del n}}\N}\simeq\lambda_{\M}\Del_{\C^{\Del (n+2)}}\lambda_{\N}$, or which is the same:
\vspace{-0,3cm}
\begin{multline*}
(\boxtimes_{\C^{\Del (n+2)}})_{j=2}^{n+2}(\boxtimes_{\C^{\Del (n+2)}})_{i=1}^{j-1} \ev_{(\M\Del_{\C^{\Del n}}\N)_{ij}}\simeq\\
(\boxtimes_{\C^{\Del (n+2)}})_{j=2}^{n+2}(\boxtimes_{\C^{\Del (n+2)}})_{i=1}^{j-1} \ev_{\M_{ij}} \Del_{\C^{\Del (n+2)}}
(\boxtimes_{\C^{\Del (n+2)}})_{j=2}^{n+2}(\boxtimes_{\C^{\Del (n+2)}})_{i=1}^{j-1} \ev_{\N_{ij}}.
\end{multline*}
To see this it is equivalent to prove that $\ev_{\M\Del_{\C}\N} \simeq \ev_{\M} \Del_{\C} \ev_{\N}$ for $\M, \N\in\dul\Pic(\C)$. The functor
$\ev_{\M\Del_{\C}\N}$ is defined through the commutative diagram $\langle 3\rangle$ below:
$$\scalebox{0.8}{
\bfig \hspace{-0,4cm}
\putmorphism(2550,600)(-1,0)[`\M\Del_{\C}\N\Del_{\C}(\M\Del_{\C}\N)^{op}` \simeq]{800}{-1}a
\putmorphism(-200,600)(1,0)[\M\Del_{\C}\M^{op}\Del_{\C}\N\Del_{\C}\N^{op}` \M\Del_{\C}\N\Del_{\C}\N^{op}\Del_{\C}\M^{op} `
	\M\Del_{\C}\tau_{\M^{op},\N\Del_{\C}\N^{op}}]{2200}1a
\putmorphism(2400,600)(0,-1)[`` \M\Del_{\C}\ev_{\N}\Del_{\C}\M^{op}]{450}1l
\putmorphism(2980,540)(-2,-1)[`\M\Del_{\C}\C\Del_{\C}\M^{op}` ]{800}0a
\putmorphism(-560,600)(3,-1)[``\Id\Del_{\C}\ev_{\N}]{1280}1r
\putmorphism(-200,540)(2,-1)[`\M\Del_{\C}\M^{op}\Del_{\C}\C`]{780}0l
\putmorphism(920,140)(1,0)[`\M\Del_{\C}\M^{op}`\simeq]{470}1a
\putmorphism(1620,140)(1,0)[``\simeq]{220}1a
\putmorphism(-840,560)(2,-1)[``\ev_{\M}\Del_{\C}\ev_{\N}]{2160}1l  
\putmorphism(1340,130)(0,-1)[`\C`\ev_{\M}]{650}1r
\putmorphism(3650,580)(-2,-1)[``]{2220}1r
\putmorphism(3650,560)(-2,-1)[``\ev_{\M\Del_{\C}\N}]{2230}0r
\put(900,360){\fbox{1}}
\put(900,-160){\fbox{2}}
\put(1700,-160){\fbox{3}}
\efig}
$$
The diagram $\langle 1\rangle$ commutes by naturality of the braiding
$\tau$ in $(\C^{br}\x\Mod, \Del_{\C}, \C)$, and $\langle 2\rangle$ commutes obviously. Then the commutativity of the outer diagram yields:
$\ev_{\M} \Del_{\C} \ev_{\N}\simeq\ev_{\M\Del_{\C}\N}(\M\Del_{\C}\tau_{\M^{op},\N\Del_{\C}\N^{op}})$, or $\ev_{\M\Del_{\C}\N}\simeq\ev_{\M} \Del_{\C} \ev_{\N}$
up to the switch isomorphism functor.
\end{rem}

Observe that we also have:
\begin{equation}\eqlabel{delta delta ev}
\end{equation} \vspace{-1,3cm}
\begin{eqnarray*}
\delta_{n+1}\delta_{n}(\ev_{\M})&\stackrel{\equref{double delta F}}{\simeq}&
 (\boxtimes_{\C^{\Del (n+2)}})_{j=2}^{n+2}(\boxtimes_{\C^{\Del (n+2)}})_{i=1}^{j-1} (\ev_{\M_{ij}}\boxtimes_{\C^{\Del (n+2)}}\crta\ev_{\M_{ij}})\\
&\simeq& \left((\boxtimes_{\C^{\Del (n+2)}})_{j=2}^{n+2}(\boxtimes_{\C^{\Del (n+2)}})_{i=1}^{j-1} \ev_{\M_{ij}}\right) \boxtimes_{\C^{\Del (n+2)}}
  \left((\boxtimes_{\C^{\Del (n+2)}})_{j=2}^{n+2}(\boxtimes_{\C^{\Del (n+2)}})_{i=1}^{j-1} \crta\ev_{\M_{ij}}\right)\\
&=& \lambda_{\M} \Del_{\C^{\Del (n+2)}}\crta\lambda_{\M}
\end{eqnarray*}

For a consequence of \leref{alfa dagger}, 1), we have:

\begin{cor} \colabel{lambda alfa} \clabel{lambda alfa}
For any $\M\in\dul{\Pic}(\C^{\Del n})$ and a $\C^{\Del n}$-bimodule equivalence $\alpha: \M\to \C^{\Del n}$ it is $\lambda_{\M}\simeq\delta_{n+1}\delta_{n}(\alpha)$.
\end{cor}

\begin{proof}
We compute:
\begin{eqnarray*}
\lambda_{\M}&=&(\boxtimes_{\C^{\Del (n+2)}})_{j=2}^{n+2}(\boxtimes_{\C^{\Del (n+2)}})_{i=1}^{j-1} \hspace{0,2cm} \ev_{ij}\\
&\simeq& (\boxtimes_{\C^{\Del (n+2)}})_{j=2}^{n+2}(\boxtimes_{\C^{\Del (n+2)}})_{i=1}^{j-1} \hspace{0,2cm} \alpha_{ij} \Del_{\C}(\alpha^{op})^{-1}_{ij}\\
&\stackrel{\equref{double delta F}}{\simeq} & \delta_{n+1}\delta_{n}(\alpha).
\end{eqnarray*}
\qed\end{proof}

When $\M=\C^{\Del n}$ and $\alpha=\Id$, one gets:
\begin{equation}\eqlabel{lambda C}
\lambda_{\C^{\Del n}}=\C^{\Del (n+2)}.
\end{equation}

\medskip

For an invertible object $X$ in $\C^{\Del n}$ set $\tilde\delta_n(X):=
\odot_{i=1}^{n+1} \hspace{1,4mm} e_i^n(X)^{(-1)^i}$, where $\odot=\Del_{\C^{\Del (n+1)}}$. 

\medskip

The following lemma we will use in the proof of \thref{VZ}:

\begin{lma}
Let $\C$ be braided and let $\M$ be a one-sided $\C^{\Del n}$-bimodule category. For an invertible object $X$ in $\C^{\Del n}$ 
we denote by $m(X)$ the $\C^{\Del n}$-bilinear
autoequivalence of $\M$ given by acting by $X$, that is $m(X)(M)=M\crta\ot X=M\Del_{\C^{\Del n}}X$ for all $M\in\M$. Then
\begin{equation}\eqlabel{m-delta} 
\delta_n(m(X))=m(\tilde\delta_n(X)),
\end{equation}
where $\delta_n$ is from \equref{delta_n Pic}.
\end{lma}

\begin{proof}
That $m(X)$ is $\C^{\Del n}$-bilinear it follows from the fact that $\C$ is braided. The left linearity is clear. The natural isomorphism
$s_{M,Y}: m(X)(M\crta\ot Y)\to m(X)(M)\crta\ot Y$ for any $M\in\M$ and $Y\in\C^{\Del n}$ is induced by
$M\crta\ot\Phi_{Y,X}: M\crta\ot Y\crta\ot X\to M\crta\ot X\crta\ot Y$. 
Then the coherence diagram for
$s_{M,Y}$ holds by one of the braiding axioms. Set $X=X^1\Del X^2\Del\cdots\Del X^n$. First note that
$\tilde\delta_n(X)=X_1\Del_{\C^{\Del (n+1)}}X_2^{-1}\Del_{\C^{\Del (n+1)}}\cdots\Del_{\C^{\Del (n+1)}}X_{n+1}^{(-1)^n}$, where we applied
\rmref{tensor product copies of C} for the tensor product in $\C^{\Del (n+1)}$. Hence we have:
$$\begin{array}{rl}\vspace{0,1cm}\hspace{-0,2cm}
m(\tilde\delta_n(X))\hspace{-0,2cm}\hskip-1em&= m(I\Del X)\Del_{\C^{\Del (n+1)}}\cdots\Del_{\C^{\Del (n+1)}} m(X^1\Del\cdots\Del X^{i-1}\Del I\Del X^i\Del\cdots\Del X^n)^{(-1)^{i-1}}\\ \vspace{0,3cm}
&\hspace{0,5cm}\Del_{\C^{\Del (n+1)}}\cdots\Del_{\C^{\Del (n+1)}} m((X\Del I)^{(-1)^n})\\ \vspace{0,1cm}
&=m(X\Del_{\C^{\Del n}} {}_{e_1^n}\C^{\Del (n+1)})\Del_{\C^{\Del (n+1)}}\cdots\Del_{\C^{\Del (n+1)}}
m(X^{(-1)^{i-1}}\Del_{\C^{\Del n}} {}_{e_i^n}\C^{\Del (n+1)})\\ \vspace{0,2cm}
& \hspace{0,5cm}\Del_{\C^{\Del (n+1)}}\cdots\Del_{\C^{\Del (n+1)}}m(X^{(-1)^n}\Del_{\C^{\Del (n+1)}}{}_{e_{n+1}^n}\C^{\Del (n+1)})\\
\vspace{0,2cm}
&=m(X)_1\Del_{\C^{\Del (n+1)}}\cdots\Del_{\C^{\Del (n+1)}} m(X)_i^{(-1)^{i-1}}\Del_{\C^{\Del (n+1)}}\cdots\Del_{\C^{\Del (n+1)}} m(X)_{n+1}^{(-1)^n}\\
&=\delta_n(m(X)).
\end{array}$$
\qed\end{proof}

\medskip

We define $\dul{Z}^n(\C,\dul{\Pic})$ to be the category with objects $(\M,[\alpha])$,
with $\M\in \dul{\Pic}(\C^{\Del n})$, and $\alpha:\ \delta_n(\M)\to \C^{\Del(n+1)}$ an equivalence of $\C^{\Del(n+1)}$-module categories so that $\delta_{n+1}(\alpha)\simeq\lambda_{\M}$. 
(Observe that the latter means that $[\delta_{n+1}(\alpha)]=[\lambda_{\M}]$ and in the category $\dul{\Pic}(\C^{\Del n})$ morphisms determined by $\delta_{n+1}(\alpha)$ and $\lambda_{\M}$ are equal). 
A morphism $(\M,[\alpha])\to (\N,[\beta])$ is an equivalence of $\C^{\Del n}$-module categories $F:\ \M\to \N$ such that $\beta\circ \delta_n(F)\simeq\alpha$.
Then $\dul{Z}^n(\C,\dul{\Pic})$ is a symmetric monoidal category, with tensor
product $(\M,[\alpha])\ot (\N,[\beta])=(\M\Del_{\C^{\Del n}}\N,[\alpha\Del_{\C^{\Del (n+1)}}\beta])$ and unit object $(\C^{\Del n}, [\C^{\Del (n+1)}])$. Note that every object in this
category is invertible, so we can consider the Grothendieck group:
$$K_0\dul{Z}^n(\C,\dul{\Pic})={Z}^n(\C,\dul{\Pic}).$$
There is a strongly monoidal functor
$$d_{n-1}:\ \dul{\Pic}(\C^{\Del (n-1)})\to \dul{Z}^n(\C,\dul{\Pic}),$$
$d_{n-1}(\N)=(\delta_{n-1}(\N),[\lambda_{\N}])$. Let $B^n(\C,\dul{\Pic})$ be the subgroup of $Z^n(\C,\dul{\Pic})$,
consisting of elements represented by $d_{n-1}(\N)$, where $\N\in\dul{\Pic}(\C^{\Del (n-1)})$.
We then define:
$$ H^n(\C,\dul{\Pic})=Z^n(\C,\dul{\Pic})/B^n(\C,\dul{\Pic}).$$

\medskip

\begin{rem}
Observe that for $(\Ll, [\alpha]), (\M\Del_{\C}\N, [\beta])\in\dul Z^1(\C,\dul{\Pic})$ we have:
$$(\Ll, [\alpha])\iso(\Ll\Del_{\C}\C, [\alpha\Del_{\C}\C])\quad\textnormal{and}\quad(\M\Del_{\C}\N, [\beta])\iso(\N\Del_{\C}\M, [\beta\comp\delta_1(\tau)])$$
where the equivalences inducing the corresponding isomorphisms are the obvious ones. This clearly extends to the similar properties in
$\dul Z^n(\C,\dul{\Pic})$ for all $n\in\Nn$. Thus the omitting of these equivalence functors, which we applied so far when computing in $\dul\Pic(\C)$,
is justified also when passing to to the category $\dul Z^n(\C,\dul{\Pic})$.
\end{rem}

\section{Categorical Villamayor-Zelinsky sequence}

This section is dedicated to the construction of the infinite exact sequence of the type Villamayor-Zelinsky. Originally, it involved Amitsur
cohomology groups for commutative algebras over fields. We construct a version of it for symmetric finite tensor categories.
In the proof of \thref{VZ} we will frequently use the following result.

\begin{prop}
Given a finite tensor category $\C$ and an invertible $\C$-bimodule category $\M$, every $\C$-module autoequivalence $F\in\Fun_{\C}(\M, \M)$
is of the form $F\iso -\crta\ot X$ for some invertible object $X\in\C$.
\end{prop}

\begin{proof}
Since $\M$ is invertible, by \cite[Proposition 4.2]{ENO} any $F\in\Fun_{\C}(\M, \M)$ is of the form $F\iso -\crta\ot X$ for some $X\in\C$.
If $F$ is an equivalence, then clearly $X$ is invertible.
\qed\end{proof}

Observe that in the way the notation is fixed in the following result the morphisms $\alpha_1$ and $\beta_1$ are trivial maps, as $H^0(\C,{\Pic})$
is the trivial group ($\Pic(k)=0$).

\medskip

\begin{thm}\thlabel{VZ}
Let $\C$ be a symmetric finite tensor category. There is a long exact sequence
\begin{eqnarray}\eqlabel{VZ seq}
1&\longrightarrow&H^2(\C,\Inv)\stackrel{\alpha_2}{\longrightarrow}H^1(\C,\dul{\Pic})\stackrel{\beta_2}{\longrightarrow}H^1(\C,{\Pic})\\
&\stackrel{\gamma_2}{\longrightarrow}&H^3(\C,\Inv)\stackrel{\alpha_3}{\longrightarrow}H^2(\C,\dul{\Pic})\stackrel{\beta_3}{\longrightarrow}H^2(\C,{\Pic})\nonumber\\
&\stackrel{\gamma_3}{\longrightarrow}&H^4(\C,\Inv)\stackrel{\alpha_4}{\longrightarrow}H^3(\C,\dul{\Pic})\stackrel{\beta_4}{\longrightarrow}H^3(\C,{\Pic})\nonumber\\
&\stackrel{\gamma_4}{\longrightarrow}&\cdots \nonumber
\end{eqnarray}
\end{thm}

\begin{proof}
\u{Definition of $\alpha_n$}. Let $\crta X\in Z^n(\C,\Inv)$. Then
$(\C^{\Del (n-1)}, [m(X)])\in \dul{Z}^{n-1}(\C,\dul{\Pic})$, since
$$\delta_n(m(X))\equal{\equref{m-delta}}m(\tilde\delta_n(X))\simeq m(I^{\Del (n+1)})=\C^{\Del (n+1)}
\equal{\equref{lambda C}}\lambda_{\C^{\Del (n-1)}}.$$
If $\crta X=\crta Y$, there is an isomorphism $X\stackrel{\iso}{\to} Y$ in $\C^{\Del n}$ and hence a natural $\C^{\Del n}$-module isomorphism between the equivalences 
$m(X)$ and $m(Y)$, so that $(\C^{\Del (n+1)}, [m(X)])=(\C^{\Del (n+1)}, [m(Y)])$. 

In the case that $\crta X$ is a coboundary: $X=\tilde\delta_{n-1}(Y)$ for some $\crta \in \Inv(\C^{\Del (n-1)})$, we have that
$m(Y):\ \C^{\Del (n-1)}\to \C^{\Del (n-1)}$ is an isomorphism between
$(\C^{\Del (n-1)}, [m(\tilde\delta_{n-1}(Y))])$ and $ (\C^{\Del (n-1)},[\C^{\Del n}])$ because of \equref{m-delta}.
Then the map
$$\alpha_n([\crta X])=[(\C^{\Del (n-1)},[m(X)])]$$
is well-defined ($\alpha_n$ is clearly a group map). \\
\u{Definition of $\beta_n$}. We define
$\beta_n[(\M,[\alpha])]=[\M]$.\\
\u{Definition of $\gamma_n$}. Let $[\M]\in Z^{n-1}(\C,\Pic)$. Then there exists
a $\C^{\Del n}$-module equivalence $\alpha:\ \delta_{n-1}(\M)\to \C^{\Del n}$.
The composition $\lambda_{\M}\circ \delta_n(\alpha)^{-1}:\ \C^{\Del (n+1)}\to \C^{\Del (n+1)}$
is an equivalence of $\C^{\Del (n+1)}$-module categories, so it is equal to $m(X)$ for some invertible $X\in \C^{\Del (n+1)}$. Now:
\begin{eqnarray*}
&&\hspace*{-2cm}
m(\tilde\delta_{n+1}(X))\stackrel{\equref{m-delta}}{=}\delta_{n+1}(m(X))=\delta_{n+1}(\lambda_{\M})\circ
((\delta_{n+1}\circ\delta_n)(\alpha))^{-1}\\
&\stackrel{\equref{delta-lambda}, \hspace{0,1cm}\cref{lambda alfa}}{\simeq}&
\lambda_{\delta_{n-1}(\M)}\circ \lambda_{\delta_{n-1}(\M)}^{-1}=\C^{\Del (n+2)},
\end{eqnarray*}
so $\tilde\delta_{n+1}(X)\iso I^{\Del (n+2)}$, and $\crta X\in Z^{n+1}(\C,\Inv)$.

If $\alpha':\ \delta_{n-1}(\M)\to \C^{\Del n}$
is another $\C^{\Del n}$-module equivalence, then we have
$\lambda_{\M}\circ \delta_n(\alpha')^{-1}=m(X')$ for some invertible $X'\in \C^{(n+1)}$.
Then $\alpha'\circ\alpha^{-1}$
is a $\C^{\Del n}$-module autoequivalence of $\C^{\Del n}$, so $\alpha'\circ\alpha^{-1}=m(Z^{-1})$,
for some invertible $Z\in \C^{\Del n}$. Now we have
$$m(X')=\lambda_{\M}\circ \delta_n(\alpha')^{-1}=\lambda_{\M}\circ \delta_n(\alpha)^{-1}\circ
\delta_n(m(Z))\equal{\equref{m-delta}}m(X)\circ m(\tilde\delta_n(Z))= m(X\tilde\delta_n(Z)),$$
yielding $X'=X\tilde\delta_n(Z)$, so $[\crta X]=[\crta X']$ in $H^{n+1}(\C,\Inv)$. Thus we have a well-defined map
$Z^{n-1}(\C,\Pic)\to H^{n+1}(\C,\Inv)$ given by $[\M]\mapsto [\crta X]$ (such that $m(X)=\lambda_{\M}\circ \delta_n(\alpha)^{-1}$ for some
$\C^{\Del n}$-module equivalence $\alpha:\ \delta_{n-1}(\M)\to \C^{\Del n}$).

This map induces a map $\gamma_n:\ H^{n-1}(\C,\Pic)\to H^{n+1}(\C,\Inv)$.
Indeed, if $[\M]\in B^{n-1}(\C,\Pic)$ so that $\M=\delta_{n-2}(\N)$ for some
$[\N]\in\Pic(\C^{\Del (n-2)})$, then $\lambda_{\N}:\ \delta_{n-1}(\M)\to\C^{\Del n}$ is a $\C^{\Del n}$-module equivalence and as above we
have for some invertible $Y\in\C^{\Del (n+1)}$:
$$m(Y)=\lambda_{\M}\circ\delta_n(\lambda_{\N})^{-1} \equal{\equref{delta-lambda}}\lambda_{\M}\circ \lambda_{\delta_{n-2}(\N)}^{-1}=
m(I^{\Del (n+1)}),$$
hence $Y=I^{\Del (n+1)}$.

\u{Exactness at $H^{n-1}(\C,\dul{\Pic})$}. It is clear that $\beta_n\circ\alpha_n=1$.\\
Take $[(\M,[\alpha])]\in H^{n-1}(\C,\dul{\Pic})$ such that $\beta_n[(\M,[\alpha])]=[\M]=1$
in $H^{n-1}(\C,{\Pic})$. Then $\M\simeq\delta_{n-2}(\N)$ for some
$\N\in \dul{\Pic}(\C^{\Del (n-2)})$. The composition
$$\lambda_{\N}^{-1}\circ\alpha:\ (\delta_{n-1}\circ\delta_{n-2})(\N)\to(\delta_{n-1}\circ\delta_{n-2})(\N)$$
is a $\C^{\Del n}$-module equivalence, so it is given by $m(X)$ for some invertible $X\in\C^{\Del n}$. Then it follows:
$$m(\tilde\delta_n(X))\stackrel{\equref{m-delta}}{=}\delta_n(\lambda_{\N})^{-1}\circ\delta_n(\alpha)\stackrel{\equref{delta-lambda}}{\simeq}
\lambda^{-1}_{\delta_{n-2}(\N)}\circ \lambda_{\M}=\C^{\Del (n-1)}$$
meaning that $\tilde\delta_n(X)\iso I^{\Del (n+1)}$ and $[\crta X]\in H^n(\C,\Inv)$.

To prove that $[(\M,[\alpha])]$ is in the image of $\alpha_n$, we will need the following:

\begin{claim}
Let $\M, \alpha$ and $X$ be as above. Then $m(X)\comp\delta_{n-1}(\crta\ev_{\M})=
\lambda_{\N^{op}}\Del_{\C^{\Del n}}\alpha$. Consequently, $\crta\ev_{\M}:\ \M^{op}\Del_{\C^{\Del (n-1)}}\M\to \C^{\Del (n-1)}$ is an isomorphism
$$(\M^{op}\Del_{\C^{\Del (n-1)}}\M, [\lambda_{\N^{op}}\Del_{\C^{\Del n}}\alpha])\to(\C^{\Del (n-1)},[m(X)])$$
in $\dul{Z}^{n-1}(\C,\dul{\Pic})$.
\end{claim}

\begin{proof}
First of all, observe that $(\M^{op}\Del_{\C^{\Del (n-1)}}\M, [\lambda_{\N^{op}}\Del_{\C^{\Del n}}\alpha])=
(\delta_{n-2}(\N^{op}),[\lambda_{\N^{op}}])\ot(\M,[\alpha])\in\dul{Z}^{n-1}(\C,\dul{\Pic})$.
Now from \equref{delta delta ev} and by \rmref{lambda mult} it follows that:
$$\delta_{n-1}(\crta\ev_{\M})
\simeq\left( (\boxtimes_{\C^{\Del n}})_{j=2}^n(\boxtimes_{\C^{\Del n}})_{i=1}^{j-1} \hspace{0,2cm} \ev_{\N^{op}_{ij}} \right)
\Del_{\C^{\Del n}}
 \left( (\boxtimes_{\C^{\Del n}})_{j=2}^n(\boxtimes_{\C^{\Del n}})_{i=1}^{j-1} \hspace{0,2cm} \ev_{\N_{ij}}  \right) .$$
Thus:
\begin{eqnarray*} 
&&\hspace*{-1cm}
\delta_{n-1}(\crta\ev_{\M})(id_{\delta_{n-1}(\M^{op})} \Del_{\C^{\Del n}} \lambda_{\N}^{-1}) \\
&\simeq &  \left(\left( (\boxtimes_{\C^{\Del n}})_{j=2}^n(\boxtimes_{\C^{\Del n}})_{i=1}^{j-1} \hspace{0,2cm} \ev_{\N^{op}_{ij}}\right)
\Del_{\C^{\Del n}}
\left( (\boxtimes_{\C^{\Del n}})_{j=2}^n(\boxtimes_{\C^{\Del n}})_{i=1}^{j-1} \hspace{0,2cm} \ev_{\N_{ij}} \right)\right)\\
&& \comp \left(id_{\delta_{n-1}(\M^{op})} \Del_{\C^{\Del n}} \left((\boxtimes_{\C^{\Del n}})_{j=2}^n(\boxtimes_{\C^{\Del n}})_{i=1}^{j-1} \hspace{0,2cm} \ev^{-1}_{\N_{ij}}\right)\right)\\
&\simeq & \left((\boxtimes_{\C^{\Del n}})_{j=2}^n(\boxtimes_{\C^{\Del n}})_{i=1}^{j-1} \hspace{0,2cm} \ev_{\N^{op}_{ij}}\right)
\Del_{\C^{\Del n}} \C^{\Del n} \\
&=& \lambda_{\N^{op}}\Del_{\C^{\Del n}} \C^{\Del n}.
\end{eqnarray*}
On the other hand, we have that $\delta_{n-1}(\ev_{\M})$ is right $\C^{\Del n}$-linear, since so is $\ev_{\M}$. This implies the first identity
in the following computation:
\begin{eqnarray*}
&&\hspace*{-2cm}
m(X)\comp\delta_{n-1}(\ev_{\M})=\delta_{n-1}(\ev_{\M})(\Id_{\delta_{n-1}(\M^{op})} \Del_{\C^{\Del n}} m(X))\\
&=& \delta_{n-1}(\ev_{\M})(\Id_{\delta_{n-1}(\M^{op})} \Del_{\C^{\Del n}} \lambda_{\N}^{-1})(\Id_{\delta_{n-1}(\M^{op})} \Del_{\C^{\Del n}}\alpha) \\
 &\simeq& (\lambda_{\N^{op}}\Del_{\C^{\Del n}} \C^{\Del n})(\Id_{\delta_{n-1}(\M^{op})} \Del_{\C^{\Del n}}\alpha) \\
&=& \lambda_{\N^{op}}\Del_{\C^{\Del n}} \alpha.
\end{eqnarray*}
\qed\end{proof}

Observe that $[(\delta_{n-2}(\N^{op}),[\lambda_{\N^{op}}])]=1$ in $H^{n-1}(\C,\dul{\Pic})$. Now we have:
\begin{eqnarray*}
&&\hspace*{-2cm}
[(\M,[\alpha])]=[(\delta_{n-2}(\N^{op}),[\lambda_{\N^{op}}])][(\M,[\alpha])]\\
&=&[(\M^{op}\Del_{\C^{\Del (n-1)}}\M, [\lambda_{\N^{op}}\Del_{\C^{\Del n}}\alpha])]=
[(\C^{\Del (n-1)},[m(X)])]=\alpha_n([\crta X]).
\end{eqnarray*}

\u{Exactness at $H^{n-1}(\C,{\Pic})$}.
Let $[(\M,[\alpha])]\in H^{n-1}(\C,\dul{\Pic})$, then $\beta_n[(\M,[\alpha])]=[\M]$.
In order to compute $\gamma_n([\M])$, we choose the $\C^{\Del n}$-module equivalence $\alpha: \delta_{n-1}(\M)\to \C^{\Del n}$.
We have $\delta_n(\alpha)=\lambda_{\M}$, so:
$$\delta_n(\alpha)\circ\lambda_{\M}^{-1}=m(I^{\Del (n+1)}),$$
and thus $\gamma_n\circ\beta_n=1$.
Now, assume for $[\M]\in H^{n-1}(\C,\Pic)$ that $\gamma_n([\M])=1$. Then there is a $\C^{\Del n}$-module equivalence
$\alpha: \delta_{n-1}(\M)\to \C^{\Del n}$ such that $\delta_n(\alpha)\circ \lambda_{\M}^{-1}=m(X)$ for some invertible $X\in\C^{\Del (n+1)}$,
and we know that $\gamma_n([\M])=[X]$, so $\crta X\in B^{n+1}(\C,\Inv)$. Then $X=\tilde\delta_n(Y)$, for some invertible $Y\in\C^{\Del n}$.
Consider the $\C^{\Del n}$-module equivalence
$$\alpha'=m(Y^{-1})\circ\alpha:\ \delta_{n-1}(\M)\to \C^{\Del n}.$$
Then:
$$\delta_n(\alpha')\circ\lambda_{\M}^{-1}=\delta_n(m(Y^{-1}))\circ m(X)\equal{\equref{m-delta}}
m(\tilde\delta_n(Y^{-1}))\circ m(X)\simeq\C^{(n+1)},$$
therefore: $[(\M,[\alpha'])]\in Z^{n-1}(\C,\dul{\Pic})$, and $[\M]=\beta_n[(\M,[\alpha'])]$.\\
\u{Exactness at $H^{n+1}(\C,\Inv)$}. Take $[\M]\in H^{n-1}(\C,{\Pic})$, and choose
a $\C^{\Del n}$-module equivalence $\alpha:\ \delta_{n-1}(\M)\to \C^{\Del n}$. Then $\gamma_n([\M])=[X]$ for some invertible $X\in\C^{\Del (n+1)}$,
where $m(X)=\lambda_{\M}\circ\delta_n(\alpha)^{-1}$, and we have:
$(\alpha_{n+1}\circ\gamma_n)([\M])=[(\C^{\Del n},[m(X)])].$ It is immediate that $\alpha$ defines an isomorphism
$$(\delta_{n-1}(\M),[\lambda_{\M}])\to (\C^{\Del n},[m(X)])$$
in $\dul{Z}^n(\C,\dul{\Pic})$. It follows that $[(\C^{\Del n},[m(X)])]\in B^n(\C,\dul{\Pic})$,
and thus $\alpha_{n+1}\circ\gamma_n=1$.
\smallskip

Take $\crta X\in Z^{n+1}(\C,\Inv)$, and assume that
$$\alpha_{n+1}([\crta X])=[(\C^{\Del n},[m(X)])]=1,$$
that is,
$$(\C^{\Del n},[m(X)])\cong d_{n-1}(\N)=(\delta_{n-1}(\N),[\lambda_{\N}])$$
in $\dul{Z}^n(\C,\dul{\Pic})$ for some $\N\in \dul{\Pic}(\C^{\Del (n-1)})$. This means that there is a $\C^{\Del n}$-module equivalence
$\alpha:\ \delta_{n-1}(\N)\to \C^{\Del n}$ such that $\lambda_{\N}=m(X)\circ \delta_n(\alpha)$. Thus $\gamma_n([\N])=[\crta X]$.

\u{$\alpha_2$ is injective}. Take $\crta X\in Z^2(\C,\Inv)$, and suppose that $\alpha_2([\crta X])=[(\C,[m(X)])]=[(\C,[m(I\Del I)])]$. Then there exists a
$\C$-module autoequivalence $\alpha: \C\to\C$ such that $m(X)=m(I\Del I)\circ \delta_1(\alpha)$. Moreover, $\alpha$ is given by
tensoring by some invertible $Y\in \C$, so: $\delta_1(\alpha)=\delta_1(m(Y))=m(\tilde\delta_1(Y))$. For a consequence we get:
$\crta X=\crta{\tilde\delta_1(Y)}\in B^2(\C,\Inv)$.
\qed\end{proof}

\subsection{Examples of symmetric finite tensor categories and their Picard groups}

Since any finite tensor category is equivalent to a category of finite-dimensional representations of a finite-dimensional weak quasi Hopf algebra, \cite[Proposition 2.7]{EO},
then every symmetric finite tensor category is equivalent to a category $\Rep H$ of finite-dimensional representations of a finite-dimensional triangular weak quasi-Hopf algebra $H$.
We will say that a finite tensor category is {\em strong} if it is equivalent to $\Rep H$, where $H$ is a Hopf algebra (in the style of \cite{NR}).

It is known that every finite-dimensional triangular Hopf algebra over an algebraically closed field of characteristic zero is the Drinfel'd twist of a modified
supergroup algebra, \cite[Theorem 5.1.1]{AEG}, \cite[Theorem 4.3]{EG}. A finite-dimensional triangular Hopf algebra $H$ with an $R$-matrix $\R$ is called a
{\em modified supergroup algebra} if there exist:
\begin{enumerate}
	\item a finite group $G$;
	\item a central element $u\in G$ with $u^2=1$;
	\item a linear representation of $G$ on a finite-dimensional vector space $V$ on	which $u$ acts as $-1$.
\end{enumerate}
such that $H\iso\Lambda(V)\rtimes kG$ as the Radford biproduct Hopf algebra, where the elements in $G$ are group-like and the elements in $V$ are $(u, 1)$-primitive.
Namely, the action of $G$ on $V$ makes the exterior algebra $\Lambda(V)$ into
a $kG$-module algebra and we can construct the smash product $\Lambda(V)\# kG$. The element of $kG\ot kG$:
$$\R=\R_u=\frac{1}{2}(1\ot 1+u\ot 1+1\ot u-u\ot u)$$
is a triangular structure on $kG$ and $\Lambda(V)$ is a Hopf algebra in ${}_{kG}\M$, by defining
$$\Delta(v)=1\ot v+v\ot 1,\quad\varepsilon(v)=0\quad\textnormal{and}\quad S(v)=-v.$$
As a matter of fact, $\Lambda(V)$ is a Yetter Drinfel'd module algebra over $kG$ with the coaction induced by $\Lambda(v)=u\ot v$, so that $\Lambda(V)$ is indeed
a Hopf algebra in ${}_H ^H\YD$.
The triangular structure $\R_u$ extends to the triangular structure of $\Lambda (V)\rtimes kG$.
\par\smallskip

The Hopf subalgebra of the Radford biproduct $\Lambda (V)\rtimes kG$ which is generated
by $u$ and by the $(u, 1)$-primitive elements of $V$ is isomorphic, as a triangular Hopf
algebra, to Nichols Hopf algebra $E(n)\iso\Lambda(n)\times k\Zz_2$, where $n=dim(V)$,
with the triangular structure $\R_u$. The Nichols Hopf algebra $E(n)$ is a modified supergroup algebra
whose representation category is the most general non-semisimple symmetric finite tensor category whithout non-trivial Tannakian subcategories.
(A symmetric fusion category is called {\em Tannakian} if it is equivalent to $\Rep G$, the category of representations of
a finite group $G$.) For $n=1$ we obtain the Sweedler Hopf algebra $H_4$.
\par\medskip

The Brauer-Picard group is computed for a number of categories. In \cite{ENO} it is done for the representation category of any finite
abelian group, and in \cite{NR} this result is extended to a number of finite groups; in \cite{M4} the Brauer-Picard group is computed for the
representation category of a modified supergroup algebra, whereas in \cite{BN} it is computed for the Nichols Hopf algebra $E(n)$; we also mention \cite{GS}.

The subgroup of the Brauer-Picard group of the representation category of a modified supergroup algebra $\Lambda(V)\rtimes kG$ determined by the
one-sided bimodule categories over $\C=\Rep(\Lambda(V)\rtimes kG)$ is the Picard group of $\C$. In view of the above said, this subgroup is one
of the central protagonists of this article, as such $\C$ is the most general symmetric finite tensor category that is strong. It was computed in
\cite[Corollary 8.10]{Gr} that $\Pic(\Rep G)\iso H^2(G, k^{\times})$ for any finite group $G$.

\par\smallskip

In \cite[Proposition 3.7]{DZ} it is proved that for a finite-dimensional quasi-triangular Hopf algebra $H$ every $H$-Azumaya algebra in the Brauer group $\BM(k, H, \R)$
is exact. This allows one to conclude that $\Pic(\Rep H)=\BM(k, H, \R)$, in view of \cite[Section 3.2]{DN}.
The latter group has extensively been studied, in particular for the Radford biproduct Hopf algebras
$H=B \rtimes L$. A deep insight about the decomposition $\BM(k, B \rtimes L ,\overline{\R}) \cong \BM(k,L,\R) \times \Gal(B;{}_L\M)$, where $L$ is a quasi-triangular Hopf
algebra whose quasi-triangular structure $\R$ extends to that on $H$ (denoted by $\overline{\R}$) and $B$ is a commutative (and cocommutative) Hopf algebra in ${}_L\M$,
is given in \cite[Theorem 6.5]{CuF}. Here $\Gal(B;{}_L\M)$ is the group of $B$-Galois objects in ${}_L\M$, which are one-sided comodules over $B$ and $B$ is cocommutative
in ${}_L\M$. Observe that this is precisely the case in the modified supergroup algebras. Henceforth, we may write:
$$\Pic(\C)\iso\Pic(\Rep(\Lambda(V)\rtimes kG))\iso\BM(k,kG,\R_u) \times \Gal(\Lambda(V);{}_{kG}\M)$$
for every symmetric finite tensor category $\C$ that is strong. Here $\BM(k,kG,\R_u)$ is the Brauer group of $G$-graded vector spaces with respect
to the braiding induced by $\R_u$. Moreover, from the direct sum decomposition proved in \cite{Car2} we also have:
\begin{eqnarray*} 
\Pic(\C)\iso \BM(k,kG,\R_u) \times S^2(V^*)^G
\end{eqnarray*}
where $S^2(V^*)^G$ is the group of symmetric matrices over $V^*$ invariant under the
conjugation by elements of $G$. 

In particular, for the Picard group of the representation categories of triangular Hopf algebras $H_4$ and $E(n)$ mentioned above
we obtain from \cite{CC2} and \cite{VZ2}:
$$\Pic(\Rep(H_4))\iso \BW(k) \times (k,+)\quad\textnormal{and}\quad \Pic(\Rep(E(n))\iso \BW(k) \times (k,+)^{n(n+1)/2}$$
where $\BW(k)$ denotes the Brauer-Wall group (the corresponding Azumaya algebras are $\Zz_2$-graded). It is known that $\BW(\CC)\iso \Zz/2\Zz$ and
$\BW(\RR)\iso \Zz/8\Zz$. The factor $(k,+)^{n(n+1)/2}$ corresponds to $Sym_n(k)$, the group of $n\times n$ symmetric matrices under addition.

\subsection*{Acknowledgments} This work was partially supported by the Mathematical Institute of the Serbian Academy of Sciences and Arts (MI SANU),
Serbia. The author wishes to thank to C\'esar Galindo and Mart\'in Mombelli for helpful discussions.

\bibliographystyle{amsalpha}

\end{document}